\documentclass[a4paper]{article}

\usepackage{amsmath}
\usepackage{amssymb}
\usepackage{amsthm}
\usepackage{amsrefs}
\usepackage{hyperref}

 \usepackage{todonotes}

\theoremstyle{plain}%
\newtheorem{theorem}{Theorem}
\newtheorem{lemma}{Lemma}
\newtheorem{corollary}{Corollary}

\theoremstyle{remark}%
\newtheorem{remark}{Remark}%

\theoremstyle{definition}%
\newtheorem{definition}{Definition}%

\newcommand{\inn}[2]{\left\langle#1,\,#2\right\rangle}
\DeclareMathOperator{\ran}{ran}
\DeclareMathOperator{\spn}{span}

\newcommand{\grad}{\nabla}
\newcommand{\Lap}{\Delta}
\newcommand{\di}{\partial}

\newcommand{\si}{\sigma}
\newcommand{\eps}{\epsilon}
\newcommand{\ls}{\lesssim}
\newcommand{\g}{\gamma}
\newcommand{\al}{\alpha}
\newcommand{\Si}{\Sigma}

\renewcommand{\b}{\bar}

\newcommand{\Rb}{\mathbb{R}}

\newcommand{\Lc}{\mathcal{L}}

\newcommand{\W}{\mathcal{W}}
\newcommand{\Sc}{\mathcal{S}}
\newcommand{\Bc}{\mathcal{B}}
\newcommand{\Ac}{\mathcal{A}}

\newcommand{\om}{\omega}

\renewcommand{\l}{\lambda} 
\newcommand{\Sb}{\mathbb{S}} 
\renewcommand{\b}{\bar} 

\newcommand{\abs}[1]{\left\lvert#1\right\rvert}
\newcommand{\norm}[1]{\left\lVert#1\right\rVert}

\newcommand{\br}[1]{\left\langle#1\right\rangle}
\newcommand{\Set}[1]{\left\{#1\right\}}

\newcommand{\pd}[2]{\ensuremath{\tfrac{\partial#1}{\partial#2}}}

\newcommand{\md}[6]{\ensuremath{
		\ifinner
		\tfrac{\partial{^{#2}}#1}{\partial{#3^{#4}}\partial{#5^{#6}}}
		\else
		\tfrac{\partial{^{#2}}#1}{\partial{#3^{#4}}\partial{#5^{#6}}}
		\fi
}}
\newcommand{\del}[1]{\left(#1\right)}
\newcommand{\thmref}[1]{Theorem~\ref{#1}}

\newcommand{\defnref}[1]{Definition~\ref{#1}}
\newcommand{\secref}[1]{Section~\ref{#1}}
\newcommand{\lemref}[1]{Lemma~\ref{#1}}

\newcommand{\remref}[1]{Remark~\ref{#1}}

\newcommand{\corref}[1]{Corollary~\ref{#1}}

\begin{document}
	\title{On the Stability of Cylindrical Singularities of the Mean Curvature Flow}
	\author{Jingxuan Zhang
	}
	\maketitle
\begin{abstract}
	We study the rescaled mean curvature flow (MCF) of hypersurfaces that are
	global graphs over a fixed cylinder of arbitrary dimensions.
	We construct   an explicit stable manifold for the rescaled MCF
	of finite codimensions in a suitable
	configuration space. For any initial
	hypersurface from this stable manifold, we
	construct a unique global solution to the rescaled MCF, 
	and derive precise asymptotics for these solutions that are valid for all 
	time.
	Using these asymptotics, we prove asymptotic stability of 
	cylindrical singularities of arbitrary dimensions under 
	generic initial perturbations.
	
	As a by-product,
	for any flow of hypersurfaces evolving
	according to the MCF
	that enters this
	stable manifold at any time
	and first develops a singularity at
	a subsequent time,
	we give a simple proof of the uniqueness of tangent flow,
	first established by Colding and Minicozzi. 
	Moreover, in this case
	we show the unique singularity
	profile is determined by 
	the hypersurface profile
	when the flow enters the stable manifold.

	For all results in this paper, 
	there is no symmetry or solitonic assumption.
\end{abstract}

\section{Introduction}\label{sec:1}
	Consider the mean curvature flow (MCF) for a family of hypersurfaces given
	by immersions $X(\cdot, t):\Rb^{n-k}\times \Rb^{k+1} \to \Rb^{n+1}$, satisfying
	\begin{equation}\label{MCF}
		\di_t X=-H(X)\nu(X).
	\end{equation}
	In this paper, we are interested in the dynamical behaviour 
	of a solution $X$ to \eqref{MCF}, which first develops a  singularity
	at $0\in \Rb^{n+1},\,t=T>0$.

	Equation \eqref{MCF} is invariant under rotation, translation, and parabolic rescaling.
	Motivated by these symmetries, we consider the following time-dependent rescaling for a solution to \eqref{MCF} as follows:
	\begin{equation}
		\label{1.1}
		X(x,\om,t)=\lambda(t)g(t)Y(y(x,t),\om,\tau)+(0,z(t)),
	\end{equation}
	where the immersion $Y$ is defined through this relation, and 
	\begin{align}
		\label{1.2.1}
		&g(t)\in SO(n+1),\quad g(0,x')=(0,x') \text{ for every $x'\in\Rb^{k+1}$},\\
		\label{1.2.2}
		&z(t)\in\Rb^{k+1},\\
		\label{1.3}
		&\lambda(t):=\left(2\int_t^T a(t')\,dt'\right)^{1/2}, \quad a(t)>0,\\
		\label{1.4}
		&y(x,t):=\l(t)^{-1}x,\\
		\label{1.5}
		&\tau(t):=\int_0^t \lambda(t')^{-2}\,dt'.
	\end{align}
	Notice that we do not fix the parameter $\si:=(g,z,a)$, but rather regard
this as a path  to be determined in the manifold 
\begin{equation}\label{Si}
	\Si:= SO(n+1)\times \Rb^{k+1}\times \Rb_{>0}.
\end{equation}

	The condition \eqref{1.3} shows that $\lambda(t)$ is uniquely
	determined by the path $a(t)>0$. Indeed, this $\lambda(t)$ is the unique solution to the 
	Cauchy problem
	$$\dot\lambda(t)\lambda(t)=-a(t),\quad\l(T)=0.$$
	Here and in the remaining of this section, 
	the dot denotes differentiation w.r.t. the fast time $t$-variable. 
	By definition, the terminal condition on $\l$ ensures the rescaling
	\eqref{1.1} gives rise to a tangent flow $Y=Y(y,\om,\tau)$
	in the microscopic variable $y$ and slow time variable $\tau$.

	Direct computation shows $X$ of the form \eqref{1.1} solves \eqref{MCF} if and only if the pair  
	$(\si,Y)$ solves 
	\begin{equation}
		\label{1.6}
		\di_\tau Y=-H(Y)\nu(Y)-a\inn{y}{\grad_y} Y+a Y-g^{-1}\di_\tau g Y -\l^{-1} g^{-1}\di_\tau z.
	\end{equation}
	To get \eqref{1.6}, one uses the relations $\l\dot \l=-a$, $\l^2\dot y=ay$, and $\dot \tau=\l^{-2}$, which follow from \eqref{1.3}-\eqref{1.5}
	respectively.
	We call \eqref{1.6} \textit{the rescaled mean curvature flow.}
		

	The rescaled MCF \eqref{1.6} has the following family of stationary solutions:
	\begin{align}
		\label{1.7}
	&	Y_{a_0}\equiv \del{y,\sqrt{\frac{k}{a_0}}\om},\\
	\label{1.7.1}
	&\si_0\equiv (g_0,z_0,a_0)\in \Si.
	\end{align}
	Geometrically, the three components in $\si_0$ consist of
	a 
	rotation $g_0$ of the cylindrical axis, 
	a transversal translation $z_0$, 
	and a dilation by a factor of $\sqrt{k/a_0}$. 
	The pair $(\si,Y)$ corresponds to a cylinder 
	with unit radius along the $y$ axis, transformed 
	by the symmetry $\si_0$ as in \eqref{1.1}.

	\subsection{Main Results}

		In the remaining of this paper, 
	we seek maximal solution $X$ to \eqref{MCF} on the spatial domain $\Rb^{n-k}\times \Sb^k$ and the time interval $0\le t<T$, of the form
	\begin{equation}
		\label{1.8'}
		X(x,\om,t)=\l(t)g(t)\underbrace{\left(y(x,t),\del{\sqrt{\frac{k}{a(t)}}+\xi(y(x,t),\om,\tau(t))}\om \right)}_{\text{as $Y$ in \eqref{1.1}}}+(0,z(t)).
	\end{equation} 

	In terms of the blow-up variables $y,\,\tau$ from \eqref{1.4}-\eqref{1.5},
	this amounts to finding a pair $(\si,Y)$ that solves
	the rescaled MCF \eqref{1.6} for all $\tau\ge0$.
	Here $\si=\si(\tau)$ is a path of parameters. $Y=Y(\xi)$ is a flow
	of graphs over  a 
	fixed cylinder,  parametrized by a path
	of functions $\xi(\cdot,\tau):\Rb^{n-k}\times \Sb^k\to\Rb_{\ge0},\,\tau\ge0$,
	as in \eqref{1.8'}.

	With this convention as well as the rescaling \eqref{1.2.1}-\eqref{1.5} understood,
	the main result of this paper is the following assertions about the rescaled MCF \eqref{1.6}:
	\begin{theorem}
		\label{thm1}
		Let $X^s(a),\,s\ge2,\,a>0$ be the weighted Sobolev  space defined in \eqref{X}.
		There exists $0<\delta\ll1$ s.th. the following holds:
		\begin{enumerate}
			\item (Global existence)
		For every $a_0\ge1/2+2\delta$, there exists
		 a subspace 
		$\Sc\subset X^s(a_0)$ with finite codimensions, together with
		a map 
		$$\Phi:\Bc_\delta\cap \Sc\to X^s\equiv X^s(1/2),\quad \Bc_\delta,\Sc\text{ as in \defnref{defn3.2}},$$
		satisfying
		\begin{align}
		\label{1.9.1}
\norm{\Phi(\eta_0)}_s&\ls \norm{\eta_0}_{X^s}^2,\\
\label{1.9.2}
\norm{\Phi(\eta_0)-\Phi(\eta_1)}_{X^s}&\ls \delta
\norm{\eta_0-\eta_1}_{X^s},
		\end{align} for every $\eta_0,\,\eta_1\in \Bc_\delta\cap\Sc$,
		as well as the following properties: 
		
		For every $\eta_0\in \Bc_\delta\cap \Sc$, 
		there exists a unique global solution
		$(\si,Y)$ to the rescaled MCF \eqref{1.6}
		with initial configuration $$Y\vert_{\tau=0}=(y,(\sqrt{k/a_0}+\eta_0+\Phi(\eta_0)\om)\quad$$
		on $\Rb^{n-k}\times \Sb^{n-k}\times \Rb_{\ge0}$. 
		
		Moreover, this solution $Y$ is uniquely determined by  the decomposition
		\begin{equation}\label{1.9.3}
			Y=(y,(\sqrt{k/a}+\xi(y,\om,\tau))\om)\quad (y\in\Rb^{n-k},\,\om\in\Sb^k),
		\end{equation}
		where $a=a(\tau)$ is a component of the path $\si(\tau)=(g(\tau),z(\tau),a(\tau))\in\Si$, and $\xi=\xi(\cdot,\tau)$ 
		is a path of  functions on $\Rb^{n-k}\times \Sb^{n-k}$.
		The path $(\si,\xi)$ lies in the space
		$$(\si,\xi)\in Lip([0,\infty),\Si)\times (C([0,\infty),X^s)\cap C^1([0,\infty),X^{s-2})).$$
		
		\item (Effective dynamics)
		Moreover, 
		the path  $\si$ evolves according to the
		following  system of ODEs:
		\begin{align}
			\label{1.11}
			&\di_\tau \si=\vec F(\si)\xi+\vec M(\si,\xi)\quad \text{ for a.e. } \tau,\\
			\label{1.12}
			&\si(0)=(\delta_{ij},0,a_0)\in\Si,
		\end{align}
		where the vector fields $\vec F,\,\vec M$ are as in \eqref{3.5}.
		
		\item (Dissipative estimates)
		Moreover, the path $\si(\tau)$ dissipates to $\si(0)$, 
		with the decay estimate 
		\begin{equation}\label{1.13}
		\abs{\si(\tau)-\si(0)}\ls\delta\br{\tau}^{-1},\quad \tau\ge0.
		\end{equation}
		Here and below, we write $\br{\cdot}:=(1+\abs{\cdot}^2)^{1/2}$.

		Moreover, the function $\xi$ in the decomposition \eqref{1.9.3}
		is non-negative for all $\tau$, and dissipates to zero, 
		with the decay estimate 
		\begin{equation}
			\label{1.14}
			\norm{\xi(\cdot,\tau)}_{X^s}\le\delta\br{\tau}^{-2},
			\quad \tau\ge0.
		\end{equation}

		\end{enumerate}
		\end{theorem}
	
		The result 
		above are obtained by studying 
		a quasilinear PDE \eqref{2.1}
		in $\xi$, coupled to a system of ODEs, \eqref{2.17}-\eqref{2.15}
		in the parameter $\si$. Up to a rigid motion
		and a dilation, initial configuration \eqref{1.12}
		can be replaced with any $\si_0\in\Si$.

		In \cite{MR2993752}*{Thm. 4.31}, Colding and Minicozzi have shown
		that the cylindrical singularity of the MCF is  $F$-unstable.
		In terms of the PDE \eqref{2.1},
		this means that the static solution
		$\xi=0$ is \textit{linearly unstable.}
		In view of this,
		\thmref{thm1} above states 
		that for a generic (i.e. finite codimensional)
		class of initial perturbations, 
		$\xi=0$ is \textit{asymptotically stable} under the evolution 
		\eqref{2.1}.
		This gives a justification of the 
		generality of cylindrical singularity
		based on the PDE ground.
		
		Indeed, at the intuitive level,
		the fact that cylinders are  $F$-unstable 
		has to do with the symmetry of the (rescaled) MCF.
		The broken symmetries of the cylinder,
		e.g. transversal translation to the axis,
		generates some unstable modes 
		to the linearized operator 
		at the cylinder,
		$$L(a):=-\Lap_y+a\inn{y}{\grad_y(\cdot)}-\frac{a}{k}\Lap_{\om}-2a.$$
		Here $a>0$ corresponds to the cylindrical radius.
		In view of this, the space $\Sc$ in \thmref{thm1},
		defined in \defnref{defn3.2} below,
		removes all the zero-unstable modes of $L(a)$
		from the configuration space.
		
		Then, as customary in the analysis of solitons,
		one would like to write down a number of 
		ODEs, 
		which governs the motion
		of a frame of symmetries
		that eliminates the zero-unstable modes of $L(a)$ for all subsequent time.
		In our situation, this gives rise to the  ODEs \eqref{1.11}-\eqref{1.12} above.
		These ODEs are  called the \textit{modulation equations},
		and their use in the study of
		solitary wave dynamics dates back to \cites{MR1071238,MR1170476}
		
		The twist here is that not all of the zero-unstable modes 
		of $L(a)$ can be eliminated by the modulation method.
		First, there is the  translation  along the cylindrical axis,
		which incurs unstable modes protected by the symmetry of the cylinder.
		Secondly, as pointed out in \cite{MR3374960}*{Sect. 3.2},
		there are certain zero modes that are not due to 
		any symmetries. 
		
		To eliminate
		these,
		following \cite{MR2480603},
		 we introduce
		the second order correction $\Phi$, constructed explicitly in \defnref{Phi} below. 
		In \cite{MR2480603}, Schlag has
		studied a $L^2$-supercritical
		nonlinear Schr\"odinger equation,
		where some unstable mode
		arises which is not due to any symmetry.
		Here we implement  Schlag's method,
		which is essentially a  renormalization technique,
		to study the rescaled MCF \eqref{2.1}.

		Let $\Phi$ be the map constructed in \thmref{Phi}.
		In terms of invariant manifold theory,
					the set
		\begin{equation}\label{M}
			M=M_{\delta,a_0}=\Set{X:(x,\om)\mapsto (x,(\sqrt{k/a_0}+\eta+\Phi(\eta))\om)\mid\eta\in\Bc_\delta\cap \Sc}
		\end{equation}
		can be viewed as a stable manifold for \eqref{MCF},
		or, more appropriately, for the rescaled MCF \eqref{1.6}.	
		We discuss the geometric property of $M$ in \secref{sec:6}.
		In this paper, our focus is on the ``shadow''
		of  a flow generated by 
		an initial configuration in $M$ under the rescaled
		MCF, namely \eqref{1.11}-\eqref{1.12},
		and dissipative estimates of the form \eqref{1.13}-\eqref{1.14}.
		Indeed, $M$ is a stable manifold precisely in the sense that
		the infinite dimensional flow of hypersurfaces is well-approximated by
		the finite system of ODEs \eqref{1.11}-\eqref{1.12}
		in the space $\Si$.

		In physical terms, \thmref{thm1} 
		is an adiabatic theory for MCF near a cylindrical singularity.
		In the physics literature, especially in classical field theory, 
		the method of \textit{adiabatic approximation} involves 
		the decomposition of a  flow 
		into a slowly  evolving main part,
		together with a small fluctuation.
		In our situation, for a rescaled flow starting
		in $M$,
		the adiabatic approximation are the ODEs \eqref{1.11},
		which govern the motion of a main part (a transformed cylinder
		by symmetry)
		moving at the speed of $O(\delta)$.
		The fluctuation is the 
		remainder $\xi$.

In view of the original MCF, 
the adiabatic dynamics \eqref{1.11}-\eqref{1.12} of $\si$ implies the following:
\begin{corollary}[uniqueness of tangent flow, c.f. \cite{MR3374960}*{Thm. 0.2}]
	\label{cor1}
	Fix $0<\delta\ll1$ and $a_0\ge1/2+2\delta$.
	Suppose $X(\cdot,t)$ is 
	 a maximal solution to \eqref{MCF} that first develops
	a singularity at $(0,T)\in\Rb^{n+1}\times\Rb_{>0}$.
	Suppose up to a  dilation by a positive factor,
	at some $t_0<T$
	the hypersurface $X(\cdot,t_0)$
	lies in the stable manifold $M_{\delta,a_0}$ given in \eqref{M}.
	Then
	the tangent flow of $X$ is unique.
	
	Moreover, 
	the limit cylinder is parametrized by 
	$(y,\sqrt{k/a_0}\om)$ with $y\in\Rb^{n-k}$ and $\om\in \Sb^k$.

\end{corollary}

The uniqueness of tangent flow  is established in
\cites{MR3349836,MR3374960,choi2021ancient} under various assumptions.
In these papers, the authors show that \textit{assuming}
a flow develops a cylindrical singularity of a certain 
profile, then the latter is independent of the tangent flow.

The point of \corref{cor1} is that for a flow starting from $M$,
we have
explicit information of the limit of the tangent flow given the initial profile. 
In fact, the uniqueness of tangent flow 
is a by-product of the equations \eqref{1.11}-\eqref{1.12}.
By assumption, \corref{cor1} applies only to
surfaces that at some $t_0<T$
that are $X^s(a_0)$-close to a cylinder.
This is as expected for an adiabatic theory,
whose validity
relies on the fact that the evolving 
configurations are near equilibira.

		\thmref{thm1} and the corollary are proved in \secref{sec:3.2}.
	
	\subsection{Historical Remarks}
In a series of papers \cites{MR2465296,MR3397388,MR3803553,MR4303943,
		gang2021dynamics,zhou2021nondegenerate},
	Zhou Gang and several co-authors have developed an adiabatic theory for
	cylindrical singularity of certain dimensions.
	Here, as remarked earlier, by \textit{adiabatic theory} we mean
	the method of slow motion approximation.
	Indeed, the key technique in 
	those papers  is the  decomposition of a solution evolving
	according to
	\eqref{MCF} into a slowly evolving main part of the form
	\begin{equation}\label{V}
			V_{a,B}=\sqrt{\frac{k+\tfrac{1}{2}\inn{By}{y}}{a}},
	\end{equation}
	together with a remainder. Here $a(t)>0$, and 
	$B(t)$ is a path of symmetric real matrix of size $n-k$ to be determined.
	(These play the role of $\si$ in \eqref{1.8'}.)
	Then one studies \eqref{MCF} in a suitable configuration space, and tries to derive
	dissipative estimate for the remainder.
	This line of thought is an essential motivation to the present paper.
	
	Ultimately, the 
	adiabatic theory of \cites{MR2465296,MR3397388,MR3803553,MR4303943,
		gang2021dynamics,zhou2021nondegenerate}
	has to do with the zero-unstable modes of 
	the linearized operator at a cylinder of radius $a>0,$
	as mentioned earlier.
Indeed, consider the expression \eqref{V}. 
	This $V$ solves the static equation
	$$-a\inn{y}{\grad_y}v+av-kv^{-1}=0,$$
	which is obtained by ignoring all time derivatives 
	and all second order spatial derivatives in
	the graphical form of the rescaled MCF, \eqref{2.1}.
	(In multi-scale analysis and classical field theory,
	solution of such truncated equation is often called \textit{the adiabatic solution.})
	For  $\abs{B}\ll1$,
	expanding $V$ in $y$ gives $V=\sqrt{k/a}(1+\tfrac{1}{4k}\inn{By}{y}+O(\abs{B}^2\abs{y}^4))$,
	where the first term corresponds to the unstable modes with eigenvalue $-2a$
	of the operator $L(a)$, and the 
	second term, a homogeneous quadratic polynomial in $y$,
	corresponds to the non-symmetry zero modes mentioned earlier. 
	This expansion of $V$ is valid for $\abs{\inn{By}{y}}\ll1$,
	which holds on a bounded region for sufficiently small $\abs{B}$.
	On such a region, it is possible to write down a system of modulation equations
	for $B$ so as to eliminate the effect of the non-symmetry zero modes.
	The analysis of the dynamics of the matrix $B$
	constitutes the most difficult part  in \cites{MR2465296,MR3397388,MR3803553,MR4303943,gang2021dynamics,zhou2021nondegenerate}.
	Then there is another modulation equation for $a$ to eliminate 
	the scaling instability, together with some 
	center manifold analysis to control the rest of zero-unstable modes.
	The authors of the above mentioned papers have developed an extension-improvement
	method (what they call the ``bootstrap machines'') to make this rigorous.
	This method culminates in a proof of the mean convexity conjecture with $\Rb^3\times \Sb^1$ singularity in \cite{MR4303943}.
	
	In the past decade, in a series of papers  \cites{MR2993752,MR3374960,MR3349836,MR3381496,MR3489702,MR3521319,MR3602529,MR3662439,MR4176547,chodosh2020mean,chodosh2021mean,choi2021ancient,choi2021ancient2}, a group of prominent geometers
	have developed 
	some remarkable regularity results for cylindrical singularities,
	using various methods of geometric analysis and geometric measure theory.
	Here we single out \cite{MR3374960}, in which Colding and Minicozzi
	developed the \L ojasiewicz inequalities for the 
	rescaled MCF. Some estimates from that paper are
	used for the developments in Sects. \ref{sec:2}-\ref{sec:6},
	though the \L ojasiewicz inequalities themselves are not used.
	
	The parabolic rescaling \eqref{1.3}-\eqref{1.5} sends a fixed 
	spacetime neighbourhood on which the MCF \eqref{MCF} is posed
	to a growing region in the blow-up
	variables $y\in\Rb^{n-k}$ and $\tau\ge0$, on which \eqref{1.6} is posed.
	This results in some unfavourable issues due to non-compactness
	in the analysis of the latter equation.
	To overcome these issues, many results in \cites{MR2993752,MR3374960,MR3349836,MR3381496,MR3489702,MR3521319,MR3602529,MR3662439,MR4176547,chodosh2020mean,chodosh2021mean,choi2021ancient,choi2021ancient2}
	are set in the Gaussian weighted Sobolev space (e.g. \cites{MR2993752,MR3374960}).
	This, together with the extensive use of geometric 
	measure theory, makes it difficult to apply 
	the regularity theory from these papers 
	to derive precise asymptotics of the form \eqref{1.11}-\eqref{1.14}
	for the rescaled MCF,
	which are explicit and valid uniformly in time.

In \cites{MR2465296,MR3397388,MR3803553,MR4303943,gang2021dynamics,zhou2021nondegenerate},
	in order to derive precise and uniform asymptotics, 
	Zhou Gang and various co-authors study the rescaled MCF
	in the blow-up variables in 
	a weighted $C^k$-space, with a norm of the form 
	$\norm{\phi}=\norm{\br{y}^{-l}\di_y^\al\di_\om^\beta\phi}_{L^\infty}$
	for a function $\phi$ defined on a fixed cylinder, 
	multi-indices $\abs{\al}+\abs{\beta}\le k$,  $\br{y}:=(1+\abs{y}^2)^{1/2}$, and
	some power $l\ge1$.
	Using this norm, these authors are able to obtain
	$C^2$-control over a frowing region of size roughly $\tau^{1/2}$ at time $\tau$
	in the microscopic scale.
	This amounts to $C^2$-control over a fixed spatial neighbourhood in
	the original scale.
	To get such control, the authors of the aforementioned papers
	have developed several novel propagator estimates
	in the polynomially weighted norm given above.
	This kind of propagator estimates are used in 
	the present paper.
	
	In contrast, 
	using the \L ojasiewicz inequalities from \cite{MR3374960},
	one has $C^2$-control only 
	over a region of size $\sqrt{\log\tau}$ on the rescaled hypersurface,
	which amounts to a shrinking region in the original, macroscopic scale.
	This limitation is due to the fact that control
	over Gaussian weighted Sobolev norm
	is, in general, not sufficient to conclude any 
	uniform control, since one loses
	almost all asymptotic data.
	In particular, we lack embedding theorems
	of Gaussian weighted Sobolev spaces into H\"older spaces.
	
	On the other hand, 
	since the adiabatic solution $V$ in \eqref{V}
	is
	not  in fact a stationary solution to the rescaled MCF, 
	the regularity theory of Colding-Minicozzi,
	based on the analysis of Huisken's $F$-functional
	(the Gaussian weighted volome), is not very useful
	for the result in \cites{MR2465296,MR3397388,MR3803553,MR4303943,gang2021dynamics,zhou2021nondegenerate}.
	Indeed, unlike the cylinder, this $V$ from \eqref{V} has nothing to do
	with the $F$-functional.
	This fact results in considerable technical complications in those papers.
	
	In  the present paper, we are only interested in the asymptotic
	singularity profile of \eqref{MCF}, and therefore
	all the analysis in the subsequent sections
	are set in the Gaussian weighted Sobolev space.
	By the first \L ojasiewicz inequality from \cite{MR3374960}*{Thm. 2.54, Thm. 5.3},
	it is not hard to see that the remainder estimate \eqref{1.14}
	implies  $C^2$-control over a region of size $\sqrt{\log\tau}$ on
	the rescaled hypersurfaces,
	but this does not correspond to a fixed region on the
	original scale. Hence, we do not conclude any uniform estimate 
	in this paper, but leave this to some future work.

	Recently, invariant manifold theory for MCF is studied in \cites{sun2021initial,sun2021initial2}. In these papers, 
the authors have constructed stable/central/unstable manifolds
	for MCF near a closed singularity, and obtained some results
	about unstable initial perturbation for conical singularity.
	These works are along the line of \cites{MR1000974,MR2439610,MR1445489,MR1675237,MR1160925},
	in which
	invariant manifold theory for semilinear parabolic
	equations are studid extensively. 
	In this paper we use  different methods and have rather different focus,
	as mentioned earlier.
	
	\subsection{Organization of the Paper}
	The purpose of the present paper is to lay out 
	a framework in which one has
	precise asymptotics and effective dynamics near a 
	cylindrical singularity of arbitrary dimensions.

	Uniform control (in time) of this kind is achieved in the main result, \thmref{thm1}, 
		for a generic class
	of initial configurations. 
	The latter is the 
	 manifold $M$ given in \eqref{M}, constructed explicitly
	in \defnref{Phi}. This is generic in the sense that it has finite
	codimensions in the configuration space. 
	
	In this paper we mostly consider the rescaled equation \eqref{1.6}.
	To keep track of the precise dynamics of the hypersurfaces,
	we restrict ourselves to graphs over 
	cylinder.
		To use some of the regularity theory of  Colding-Minicozzi,
	we study the rescaled MCF in the Gaussian weighted 
	Sobolev space $X^s(a)$.
	The graphical equation thus obtained and the configuration spaces are 
	defined in \secref{sec:2}.
	
	In \secref{sec:2.2}, we collect some known results regarding 
	the linearized operator $L(a)$ at a cylinder. 
	In \secref{sec:2.3},
	we develop the first order correction that eliminates
	the destabilizing effect of  the zero-unstable modes arising
	from broken symmetry. This is done by imposing a certain set
	of modulation equations, namely \eqref{2.17}-\eqref{2.15}.
	
	In \secref{sec:3}, we define the key map $\Phi$
	in \thmref{thm1}. This is defined in
	terms of the unique fixed points of a family of maps
	$\Psi=\Psi(\cdot,\eta_0)$, also introduced in that section.
	In \secref{sec:3.2}, we prove the main theorem and its corollaries,
	assuming the key ingredients from Sects. \ref{sec:4}-\ref{sec:6}.
	
	Indeed, \thmref{thm1} relies on a fixed point scheme 
	for the map $\Psi$. 
	In \secref{sec:4}, we show that this map $\Psi$ goes from a suitable
	configuration space $\Ac_\delta$ into itself. 
	This section contains the heart of the matter, namely an
	appropriate second order correction 
	that eliminates the effect of the remaining zero-unstable modes
	that persist the  modulation.
	These remaining modes
	 correspond to translation symmetry $Y(y,\om)\mapsto Y(y+y',\om)$
	 along the cylindrical axis, which is not broken by the cylinder,
	and the non-symmetry zero modes mentioned earlier.
	
	In \secref{sec:5}, we show that $\Psi$ is a contraction in $\Ac_\delta$.
	In \secref{sec:6}, we prove the estimates \eqref{1.9.1}-\eqref{1.9.2},
	which show that the set $M$  in \eqref{M} forms a Lipschitz graph over $\Bc_\delta\cap \Sc$ of
	size $\delta$, and therefore can be treated as a manifold.
	
	In Appendix \ref{sec:B}, we prove some elementary results that are
	nonetheless crucial for the developments in Sects. \ref{sec:4}-\ref{sec:6}.
	In Appendix \ref{sec:A}, we prove that the map $\Psi$ from \secref{sec:3.1}
	is well-defined.
	In Appendix \ref{sec:C}, we collect certain technical nonlinear estimates.

	 	\subsection*{Notation}
	 Throughout the paper, the notation $A\ls B$ means that there is a constant $C>0$ independent of time and any small parameter 
	 in the problem, such that $A\le CB$. 
	 For two vectors $A,\,B$ in Banach spaces $X,\,Y$ respectively,
	 the notation $A=O_{Y}(B)$ means $\norm{A}_{X}\ls \norm{B}_{Y}$.

	\section{The modulation equations: first-order correction}
	\label{sec:2}
	In this section we develop the first order correction for the
	rescaled  mean curvature flow.
	
	By the relations \eqref{1.1} and \eqref{1.8'}, as far as \thmref{thm1} is concerned, it suffices to consider 
\eqref{1.6} in the unknown pair $(\si,\xi)$,
entering the equation through \eqref{1.8'}.

		There are various advantage of studying \eqref{1.6} instead of \eqref{MCF}.
		This is typical when one is interested in blow-up solutions to 
		evolution equations,  such as in  the study of 
		critical solitary wave dynamics where in general
		solution blows up with various profiles in finite time.
		See for instance \cite{MR2656972}.
		For one, \eqref{1.6} has global solutions in time.
		(Actually, \thmref{thm5.1} below proves global well-posedness
		for \eqref{1.6} for a large class of initial configurations.)
		Thus we will mostly work with \eqref{1.6}
        and only return to \eqref{MCF} in \secref{sec:3.2}, when we derive geometric
		consequences of the effective dynamics for the original flow. 
		Moreover, since \thmref{thm1} only concerns with solutions 
		that are normal graphs over $\Rb^{n-k}\times \Sb^k$ of the form \eqref{1.8'}, 
		when we study \eqref{1.6} below, we reparametrize 
		$Y$ as $Y=(x,v(y,\om,\tau)\om)$ and analyze the behaviour of the 
		path of functions $v(\cdot,\tau):\Rb^{n-k}\times \Sb^k\to \Rb,\,\tau\ge0$.

	To this end, following \cites{MR2993752,MR3374960},
	 we introduce a family of Gaussian weighted Sobolev spaces. 
	\begin{definition}
		\label{defn2.0}
		For $s\ge0,a>0$, define the space
		\begin{equation}
			\label{X}
				X^s(a):=H^s(\Rb^{n-k}_y\times \Sb^k_\om\,,\Rb;\,\rho_a),\quad \rho_a:=e^{-a\abs{y}^2/2}\,d\mu.
		\end{equation}
		Here $d\mu$ is the canonical measure on $\Rb^{n-k}\times \Sb^k$.
		
		The space $	X^s(a)$ is equipped with the weighted inner product
		\begin{equation}
			\label{inn}
			\inn{\phi}{\psi}_{a}=\int \phi\psi \rho_a\quad(\phi,\psi\in X^s(a)).
		\end{equation}
		The induced norm by this inner product on $X^s(a)$
		is denoted by $\norm{\cdot}_{s,a}$.
		For simplicity, we write $X^s\equiv X^s(1/2)$ with norm $\norm{\cdot}_s$.
		The space $X^s$ has standard Gaussian measure $e^{-\abs{y}^2/4}d\mu$ centered at $y=0$. (The choice of a pivot $a=1/2$ can be replaced by any other number.)
		For a linear map $L:X^s(a)\to X^r(b)$, we denote the operator 
		norm of $L$ as $\norm{L}_{s,a\to r,b}$.
	\end{definition}

	For $s\le r$ and  $0< b\le a$, there holds the trivial continuous embedding
	\begin{equation}\label{2.0}
		X^r(b)\subset X^s(a).
	\end{equation}

	In Appendix \ref{sec:B}, we consider the issue of estimating
	the weaker $X^s(b)$-norm in terms of the $X^s(a)$-norm, with $a\ge b$.

	 Following \cites{MR1030675}, 
	 we introduced the weighted volume, or  the $F$-functional, for a normal graph $v=v(y,\om)$ over cylinder  as
	 \begin{equation}\label{F}
	 	F_a(v):=\int_S\rho_a\,d\mu_S,\quad S=S(v):=\Set{(y,v(y,\om)\om):y\in\Rb^{n-k},\,\om\in\Sb^k}.
	 \end{equation}
 	Using Sobolev inequalities, one can show that
 	this functional is $C^2$ as a Fr\'echet-differentiable map on the space $X^s(a)$ with $a>0,\,s\ge2$.
 Indeed, denotes $F_a'(v)$ the $X^0(a)$-gradient of $F_a$ at $v$.
	 Explicitly, we have
\begin{equation}\label{F'}
	F_a'(v)=-\Lap_yv+a\inn{y}{\grad_y}v-v^{-2}\Lap_\om v-av+kv^{-1}+N_1(v),
\end{equation}
where $N_1(v):X^s(a)\to X^{s-2}(a)$ is a quasilinear elliptic operator, given explicitly in \eqref{N1}.
Expression \eqref{F'} can be derived  from the first variation formula in \cite{MR2993752},
and the graphical MCF equation obtained in \cites{MR1100206}.
The nonlinear map $F_a'$ is $C^1$ from $X^s(a)\to X^{s-2}(a)$.

	\subsection{The graphical equations}

	In this subsection we consider \eqref{1.6} in the space $X^s(a)$ with $a>0,\,s\ge2$.
	
\begin{itemize}
	\item 	In the remaining of this paper, unless otherwise stated, 
	the dot denotes $\pd{}{\tau}$.
\end{itemize}
	\begin{lemma}
		\label{lem2.1}
		$X$ of the form \eqref{1.8'} satisfies the MCF \eqref{MCF}  if and only if  
		the pair $(\si,\xi)$ satisfies
		\begin{equation}
			\label{2.1}
			\dot \xi=-F_a'(\sqrt{k/a}+\xi)-\di_\si W(\si)\dot \si,
		\end{equation}
	where $F_a'(v)$ is as in \eqref{F'}, and $W:\Si\to X^s(a)$ is given by
	 \begin{equation}\label{W}
		W(\si):=\sqrt{k/a}+g_{n-k+l,j}\om^ly^j+\inn{z}{\l^{-1} \om},
	\end{equation}

	\end{lemma}

	\begin{proof}
		
	Write $v=\sqrt{k/a}+\xi$.
	Recall from the Introduction
	that $X$ of the form \eqref{1.8'} satisfies the MCF \eqref{MCF}  if and only if  
	the rescaled hyersurface
	$Y=(y,v(y,\om)\om,\tau)$ from \eqref{1.1} solves \eqref{1.6}.
	Following \cites{MR2465296,MR3397388,MR3803553,MR4303943,
		gang2021dynamics,zhou2021nondegenerate},
	we derive the evolution of $v$ from \eqref{1.6}.
	
	Taking inner product on both sides of \eqref{1.6} with the vector $(0,\om)\in\Rb^{n+1}$, we find
	\begin{multline}\label{2.2}
		\dot v=-\inn{H(v)\nu(v)}{(0,\om)}-a\inn{y}{\grad_y}v+av
		-\inn{\dot g (y,v\om)}{(0,\om)}-\l^{-1}
		\inn{\dot z}{\om}.
	\end{multline}
	Here $H(v)$ and $\nu(v)$ denotes the mean curvature and the unit normal vector
	at the point $(y,v(y,\om)\om)$ respectively.
	To get \eqref{2.2}, we use the fact that $\inn{g^{-1}(\cdot)}{\cdot}=\inn{\cdot}{g(\cdot)}$ as $g\in SO(n+1)$, and the requirement $g(0,x')=(0,x')$ for every $x'\in \Rb^{k+1}$. 
	
	
	The first term in the r.h.s.  is calculated e.g. 
	in \cite{MR3803553}*{Appendix A}, whose
	 derivation is routine. 
	 This, together with the following two terms, constitute the r.h.s. of \eqref{F'}
	 
	 Differentiating the condition \eqref{1.2.1}, we find
	$\dot g(0,x')=0$. Thus the
	    penultimate term in the r.h.s. of \eqref{2.2} further simplifies as
	$$\inn{\dot g (y,v\om)}{(0,\om)}=\dot g_{n-k+l,j}\om^ ly^j\quad (1\le j\le n-k,\,1\le l \le k+1).$$
	Lastly, moving the $\tau$-derivative of $\sqrt{k/a}$ in the l.h.s. of \eqref{2.2} (recall $v(\tau)=\sqrt{k/a(\tau)}+\xi(\tau))$,
	gives \eqref{2.1}.
	\end{proof}

	\begin{remark}\label{remW}
	For future use, here we note that for every $s\ge0$ and $b>0$,
	the function $W:\Si\to X^s(b)$ is smooth and bounded
	on the region
	$$\Set{(g,z,a)\in \Si: \lambda z\le C, \,a\ge c },$$
	for any fixed  constants $c,\, C>0$. 
	\end{remark}

	\subsection{Linearized operator at $Y_a$}\label{sec:2.2}
	The only nonlinear term in \eqref{2.1} is contained in the map $F'_a$,
	which is the (normal) $X^0(a)$-gradient of the $F$-functional,
	\eqref{F}. 
	In this subsection we linearize this map around the stationary
	solution $v\equiv \sqrt{k/a}$, which corresponds to a cylinder
	with radius $\sqrt{k/a}$ and along the $y$ axis.
	Then we study the zero-unstable modes of the linearized
	operator.

	All of the results in this section are known by far. See for instance 
	\cite{MR2993752}*{Sec.5}, \cite{MR3374960}*{Sec. 3.2}.
	\begin{lemma}
		\label{lem2.2}
		The linearized operator of $F'_a(v)$
		at the critical point $v\equiv \sqrt{k/a}$ is given by
		\begin{equation}\label{2.3}
			L(a):=-\Lc_a-\frac{a}{k}\Lap_{\om}-2a,
		\end{equation}
	where $\Lc_a:=\Lap_y-a\inn{y}{\grad_y(\cdot)}$ is the drift Laplacian on the
	weighted space $X^s(a)$.

	Moreover, the operator $L(a)$ is self-adjoint in $X^s(a)$ w.r.t. the inner product from \eqref{inn}, and is bounded from $X^s(a)\to X^{s-2}(a)$.
	The spectrum of $L(a)$ is purely discrete, 
	and the only non-positive eigenvalues, together with the associated 
	eigenfuncitons, are given by
	\begin{align}
		\label{2.4}
		&-2a,\quad \text{with eigenfunction }\Si^{(0,0)(0,0,0)}(a):=-\frac{\sqrt{k}}{2}a^{-3/2},\\
		\label{2.5}
		&-a,\quad \text{with eigenfunctions }\Si^{(0,1)(0,0,l)}(a):=\l^{-1}\om,\\
		\label{2.6}
		&-a,\quad \text{with eigenfunctions }\Si^{(1,0)(i,0,0)}(a):=\frac{1}{\norm{y^i}_{0,a}^2}y^i,\\
		\label{2.7}
		&0,\quad \text{with eigenfunctions }\Si^{(1,1)(i,0,l)}(a):=y^i\om^l,\quad \\
		\label{2.8}
		&0,\quad \text{with eigenfunctions }\Si^{(2,0)(i,j,0)}(a):=\frac{1}{\norm{ay^iy^j-\delta_{ij}}_{0,a}^2}(ay^iy^j-\delta_{ij}).
	\end{align}
	Here and in the remaining of this paper, the indices $l=1,\ldots,k+1$ and $i,j=1,\ldots,n-k$.
	
	Moreover, the functions in \eqref{2.4}-\eqref{2.7} are mutually orthogonal w.r.t. the inner
	product \eqref{inn}.
	\end{lemma}
	\begin{remark}
		Notice that we do not normalize \eqref{2.4},\eqref{2.5} and \eqref{2.7}. The point is that
		these correspond to various components of the Fr\'echet derivative
		$\di_\si W(\si)$, where $W$ is as in \eqref{W}.
	\end{remark}
	\begin{proof}
	Here we only remark, after \cite{MR3374960,MR2465296, MR3397388, MR3803553,MR4303943}, that
	the eigenvalues and eigenfunctions in \eqref{2.4}-\eqref{2.8}
	can be deduced from the spectral properties of the 
	two components in \eqref{2.3}.
	Indeed, 
	the drift Laplacian $-\Lc_a\ge0$ has 
	\begin{align}
		\label{2.9}
		&\text{eigenvalue }0,\quad \text{with eigenfunction }\phi^{0,0,0}\equiv 1,\\
		\label{2.10}
		&\text{eigenvalue }a,\quad \text{with eigenfunctions }\phi^{1,i,0}:=y^i,\\
		\label{2.11}
		&\text{eigenvalue }2a,\quad \text{with eigenfunctions }\phi^{2,i,j}:ay^iy^j-\delta_{ij}.
	\end{align}
	The spherical Laplacian $\Lap_{\Sb^k}\ge0$ has 
	\begin{align}
		\label{2.12}
		&\text{eigenvalue }0,\quad \text{with eigenfunction }Y^{0,0}\equiv1,\\
		\label{2.13}
		&\text{eigenvalue }k,\quad \text{with eigenfunctions }Y^{1,l}:=\om^l.
	\end{align}
	Since the operator $L(a)$ in \eqref{2.3} is decoupled, \eqref{2.9}-\eqref{2.13}
	give the following
	eigenfunctions of $L(a)$:
	 $$\Si^{(m,n)(i,j,k)}=\phi^{m,i,j}Y^{n,l}.$$ 
	These correspond to the only non-positive eigenvalues of $L(a)$. 
	\end{proof}

	Thus, with $L_a$ given in \eqref{2.3}, we can rewrite \eqref{2.1} as 
	\begin{equation}
		\label{2.1'}
		\dot \xi= -L(a)\xi-N(a,\xi)-\di_\si W(\si)\dot \si,
	\end{equation}
	where the nonlinearity $N(a,\xi)$ is defined by this expansion.
	This map is calculated
	  explicitly in \eqref{N}.
	 In \secref{sec:C}, we study some key properties of the nonlinearity.

	\subsection{The modulation equation}
	\label{sec:2.3}
	In this subsection we study \eqref{2.1'},
	which is equivalent to \eqref{2.1} and the rescaled MCF, \eqref{1.6}.
	Equation \eqref{2.1'} is the central subject in the remaining sections.

	It is important to note that \textit{some, but not all of} the zero-unstable modes of
	the linearized operator $L(a)$
	are due to broken symmetries. Indeed, for $\Si^{(m,n)}(a)$ 
	with $(m,n)=(0,0),(0,1),(1,1)$, there exists a path $\si(s)\in \Si$
	(the manifold defined in \eqref{Si}) s.th.
	$$  \Si^{(m,n)(i,j,l)}(a) =\di_s\vert_{s=0} T_{\si(s)} Y_a.$$
	Here $T_\si=T_{g,z,a}$ denotes the action of symmetry
	on the graph function,
	 derived from
	\eqref{1.1}, and $Y_a$
	is the cylinder from \eqref{1.7}.
	For example,  take $\si(s)=(g_0,z(s),a_0)$
	with a path $z(s)\in \Rb^{k+1},\,\di_s\vert_{s=0}z(s)=\l^{-1}e^1$.
	Then 
	$$\di_s\vert_{s=0} T_{\si(s)} Y_{a_0}=\di_s\vert_{s=0}  (\sqrt{k/a_0}+z_l(s)\om^l)=\Si^{(0,1),(0,0,1)}(a_0).$$
	Consequently, with a proper choice of $\si$ as a function of $\xi$,
	we can eliminate these modes.
	This is the content of the next lemma.


	\begin{lemma}
		\label{lem2.3}
		Suppose the pair $(\si,\xi)$ is a global solution to \eqref{2.1},
		s.th. $$\inn{\xi(0)}{\Si^{(m,n)}(a(0))}_{a(0)}=0$$ for $(m,n)=(0,0),(0,1),(1,1)$.
		Then $\xi$ satisfies the orthogonality condition
		\begin{equation}
			\label{2.14}
			\inn{\xi(\tau)}{\Si^{(m,n)}(a(\tau))}_{a(\tau)},\quad \tau\ge0,\;(m,n)=(0,0),(0,1),(1,1),
		\end{equation}
		if and only if $\si=(g,z,a)$ satisfies the modulation equations:
		\begin{align}
			\label{2.17}
\begin{split}
				\norm{\Si^{(1,1)(j,0,l)}(a)}^2_{0,a}\dot g_{n-k+l,j}=&-\inn{N(a,\xi)}{\Si^{(1,1)(j,0,l)}(a)}_a,
\end{split}\\
			\label{2.16}
\begin{split}
				\norm{\Si^{(0,1)(0,0,l)}(a)}^2_{0,a}\dot z_l=&a\inn{\xi}{\Si^{(0,1)(0,0,l)}(a)}_a\\&-\inn{N(a,\xi)}{\Si^{(0,1)(0,0,l)}(a)}_a\\&+\inn{\xi}{\di_\tau \Si^{(0,1)(0,0,l)}(a)}_a,
\end{split}
\\
			\label{2.15}
\begin{split}
				\norm{\Si^{(0,0)}(a)}^2_{0,a}\dot a=&2a\inn{\xi}{\Si^{(0,0)}(a)}_a\\&-\inn{N(a,\xi)}{\Si^{(0,0)}(a)}_a+\inn{\xi}{\di_\tau \Si^{(0,0)}(a)}_a.
\end{split}
		\end{align}
	\end{lemma} 
\begin{remark}
	\label{rem7}
	The terminology \textit{modulation equations} is common in the study of solitary wave dynamics, e.g. \cites{MR2094474,MR2351368,MR1071238,MR1170476},
	and refers to equations of the form \eqref{2.17}-\eqref{2.15},
	derived from a condition as \eqref{2.14}. We adopt the same terminology
	here. 
\end{remark}
\begin{proof}
	Differentiating both sides of \eqref{2.14} w.r.t. $\tau$, we find that
	\eqref{2.14} holds for all $\tau\ge0$ if and only if \begin{enumerate}
		\item  \eqref{2.14} holds for $\tau=0$;
	\item  For $\tau>0$, there holds
	\begin{equation}\label{2.15'}
		\inn{\dot\xi(\tau)}{\Si^{(m,n)}(a(\tau))}_{a(\tau)}=-\inn{\xi(\tau)}{\di_\tau \Si^{(m,n)}(a(\tau))}_{a(\tau)}.
	\end{equation}
	\end{enumerate}
	The first point is in the assumption.
	For the second point, at time $\tau$, taking the $a(\tau)$-weighted inner product (see \eqref{inn}) of 
	both sides of \eqref{2.1'}, using the spectral property of $L(a)$ shown in \lemref{lem2.2},
	 together with the orthogonality among $\Si^{(m,n)}$ as $L(a)$ is self-adjoint in $X^s(a)$,
	we conclude that  \eqref{2.15'} is equivalent to \eqref{2.15}.
\end{proof}

	From \eqref{2.17}-\eqref{2.15} we see that if we have a priori
	control of $\xi$ in $X^s$-norm, then we can ensure that the path $\si(\tau)$
	stays approximately constant.
	This is the usual low velocity condition for an adiabatic theory.

	\section{The quadratic correction $\Phi$ and its properties}
	\label{sec:3}
	In this section we define the key map $\Phi$ in \thmref{thm1}, which originates 
	from a fixed point scheme. Then we prove \thmref{thm1}, assuming the results in Sects. \ref{sec:4}-\ref{sec:6}. These results are inspired by \cite{MR2480603} and we use
	some of the  notations there.
	
	From now on we will often work with $X^s(a)$ spaces with different 
	$a$. In order to facilitate this effort, we 
	introduce the following notion of admissible paths:
	\begin{definition}
		\label{defn2.1}
		Fix $0<\delta\ll1$. A path $\si(\tau)=(g(\tau),z(\tau),a(\tau))\in \Si$ is admissible if the following holds:
		\begin{enumerate}
			\item The map $\tau \mapsto \si(\tau )$ is continuous and locally Lipschitz.
			\item The entire path $a(\tau)$ lies in the $\delta$-neighbourhood
			of $a(0)$, i.e.
			$$
			\sup_{\tau\ge0} \abs{a(\tau)-a(0)}\le \delta.
			$$
			\item Initially, $a(0)\ge 2\delta+ 1/2$;
			\item The path $z(\tau)<\sqrt{k/a} e^{\int^\tau a}$, or equivalently $z(\tau)/\l(\tau)< \sqrt{k/a}$
			for all $\tau$. 
		\end{enumerate}
	\end{definition} 
	
	Item 1 above is a standard regularity requirement to
	make sense of the velocity of the path $\si$.
	Items 2-3, which are the most important ones,
	are imposed for several reasons.
	First,  we want to study \eqref{2.1}
	in the space $X^s\equiv X^s(1/2)$, while from time to time 
	we consider the stronger norm $\norm{\cdot}_{s,a(\tau)}$.
	These conditions ensure  $a(\tau)\ge1/2$ and therefore $X^s\subset X^s(a(\tau))$ for all $\tau$
	by \eqref{2.0}. 
	Secondly, for some estimates, e.g. to bound a  projection into
	an eigenspace of $L(a)$, we need $a$ to stay within a 
	fixed range, so that these estimates can be made independent of $a$.
	Thirdly, for some other estimates, e.g. the nonlinear estimate \eqref{A.3},
	we need $a(\tau)$ to stay a fixed distance away from the
	pivotal number $1/2$.
	Item 4 above is imposed so that the function $W$ in \eqref{W} remains bounded. 

	\subsection{The definition of $\Psi$}
	\label{sec:3.1}
 	In this subsection, we introduce another key map $\Psi$. This map
 	goes into the definition of the map $\Phi$ in \thmref{thm1}.

	Consider the following linear equation obtained by freezing coefficients in \eqref{2.1'}
	and the modulation equations \eqref{2.17}-\eqref{2.15} at a fixed path
	\begin{equation}\label{3.0}
		(\si^{(0)},\xi^{(0)})\in Lip([0,\infty),\Si)\times (C([0,\infty),X^s)\cap C^1([0,\infty),X^{s-2})),
	\end{equation}
	given as follows:
	\begin{align}
		\label{3.4}
		&\dot \xi=-L(a^{(0)})\xi -N(a^{(0)},\xi^{(0)})-\di_\si W(\si^{(0)})\dot\si,\\
		\label{3.5}
		&\dot \si = \vec F(\si^{(0)})\xi+\vec M(\si^{(0)},\xi^{(0)}),
	\end{align}
	where $\vec F,\,\vec M$ are defined by isolating the linear and nonlinear terms in
	$\xi$ respectively in \eqref{2.17}-\eqref{2.15}. For instance, 
	the last entry of $\vec F(\si^{(0)})\xi$ is given by 
	$$-\frac{1}{\norm{\Si^{(0,0)}(a^{(0)})}_{a^{(0)}}^2}\del{2a^{(0)}\inn{\xi}{\Si^{(0,0)}(a^{(0)})}_{a^{(0)}}+\inn{\xi}{\di_\tau \Si^{(0,0)}(a^{(0)})}_{a^{(0)}}},$$
	and the last entry of $\vec M(\si^{(0)},\xi^{(0)})$ is given by
	$$-\frac{1}{\norm{\Si^{(0,0)}(a^{(0)})}_{a^{(0)}}^2}\inn{N(a^{(0)},\xi^{(0)})}{\Si^{(0,0)}(a^{(0)})}_{a^{(0)}}.$$
	To \eqref{3.4}-\eqref{3.5} we associate the initial configuration
	\begin{align}
		\label{3.6}
		& \si(0)=(\delta_{ij},0,a_0)\quad \text{for some fixed }a_0>1/2,\\
		\label{3.7}
		&\xi(0)= \eta_0+\beta_i \Si^{(1,0)(i,0,0)}(a_0)+\g_{ij}\Si^{(2,0),(i,j,0) }(a_0),
	\end{align}
	where $\eta_0\in X^s(a_0)$ is fixed, and 
	and $\beta_i,\g_{ij}\in\Rb$ are to be chosen later as functions of $\eta_0$.
	
	In the remaining of this paper, the central object is the
	 Cauchy problem \eqref{3.4}-\eqref{3.7}.
	In Appendix \ref{sec:A}, we show the linear system \eqref{3.4}-\eqref{3.5}
	is equivalent to a single linear equation 
	\begin{equation}
		\label{3.8}
		\dot \xi+L(a^{(0)})\xi= -\W (\si^{(0)})\xi- \tilde N(a^{(0)},\xi^{(0)}),
	\end{equation}
	where the maps $\W,\,\tilde{N}$ are defined in \eqref{A.1}.
	We also show the Cauchy problem \eqref{3.4}-\eqref{3.7}
	 has a unique global solution in the space \eqref{3.0}.

	Now, we set up a contraction scheme in the following space:
	\begin{definition}
		\label{defn3}
		Fix $0<\delta\ll1$. 
		The space $\Ac_\delta=\Ac_\delta^\si\times \Ac_\delta^\xi$ consists of  $$(\si,\xi)\in Lip([0,\infty),\Si)\times (C([0,\infty),X^s)\cap C^1([0,\infty),X^{s-2})),$$
		s.th. the following holds:
		\begin{enumerate}
			\item $\si(0)$ is as in \eqref{3.6}, with $a_0\ge \tfrac{1}{2}+2\delta$;
			\item For some fixed $c_0>0$, there hold the decay estimates
		\begin{align}
			\label{3.1} 
			&\abs{\dot\si (\tau)}  \le c_0 \delta \br{\tau}^{-2},\quad \tau \ge0,\\ 
			\label{3.2}
			&\norm{\xi(\tau)}_s\le \delta\br{\tau}^{-2},\quad \tau \ge0;
		\end{align}
	\item 
	Let $b:=\tfrac{1}{2}-4\delta.$ 
	There holds the pivot condition from \lemref{lemB.0}, i.e. 
	\begin{equation}
		\label{3.3'}
		\norm{\xi(\tau)}^2_{s,b}\le c,\quad \tau \ge0,
	\end{equation}
for some fixed $c>0$. 
		\end{enumerate}
	\end{definition}

	Clearly, $\Ac_\delta$ is not empty, as the pair $(\si,\xi) \equiv(\si(0),0)$ lies 
	in this space.
	
	Item 1 above amounts to fixing an initial cylindrical 
	coordinate. Up to a rigid motion in $\Rb^{n+1}$ and an initial dilation,
	\eqref{3.6} can be replaced by any other parameters. 
	The decay conditions in Item 2, \eqref{3.1}-\eqref{3.2},
	are the most important ones here, for the following reasons.
	First, these correspond to the claimed decay
	properties from \thmref{thm1},
	and finding a fixed point in $\Ac_\delta$
	amounts to constructing
	a global solution to \eqref{3.8}.
	Secondly, if $\si\in\Ac_\delta^\si$,  then $\si$ is admissible as in \defnref{defn2.1},
	and we have the convenient properties mentioned earlier.
	Item 4 has to do with the interpolation from \lemref{lemB.0},
	whose importance is explained in \secref{sec:B}.
	Of course, if \eqref{3.3'} holds, 
	then $\xi(\tau)\in X^s(b)$. But we never
	use this fact other than a pivot condition for
	deriving estimates in $X^s$.

	Consider the solution map 
	\begin{equation}
		\label{3.9}
		\Psi: (\si^{(0)},\xi^{(0)})\mapsto \text{ the unique solution $(\si,\xi)$ to \eqref{3.4}-\eqref{3.7}}.
	\end{equation} 
	In \lemref{lemA.1}, we show this map is well-defined.
	Hereafter we want to show that 
	\begin{enumerate}
		\item $\Psi(\Ac_\delta)\subset \Ac_\delta$;
		\item  $\Psi:\Ac_\delta\to \Ac_\delta$ is a contraction w.r.t. a suitable norm on $\Ac_\delta$.
	\end{enumerate}
%
		In \secref{sec:4}, we show the first point holds for appropriately chosen initial
		configurations \eqref{3.7}. 
		In \secref{sec:5} we show the second point holds
	if $\delta\ll1$. 
	
	\begin{remark}\label{remP}
		Consider the initial configuration \eqref{3.7}.
		The constants $\beta_i$ and $\g_{ij}$ are to be determined later
			as a function of $\eta_0\in X^s(a_0)$.
	This way
	the map $\Psi$ depends only on the function $\eta_0$, i.e. $\Psi=\Psi(\cdot,\eta_0)$.
	
		Indeed, if $\eta_0=0$, then it is easy to see that the fixed point
	of $\Psi(\cdot,0)$ is just the vector $(\si,\xi) \equiv(\si(0),0)$ in $\Ac_\delta$.
	This corresponds to the trivial static solution of \eqref{2.1'},
	namely the cylinder of radius $\sqrt{k/a_0}$. 
	On the other hand,
	if $\eta_0\ne0$, then in general $\Psi(0,\eta_0)\ne0$ by \eqref{3.9}.
	
	 For simplicity of notation, most of the time
	we do not write
	$\Psi=\Psi(\cdot,\eta_0)$ but let this be understood.

	\end{remark}

	\subsection{The definition of $\Phi$}\label{sec:3.1.1}
	For a function $\eta$, the map $\Phi$ is defined 
	in terms of the fixed point of $\Psi=\Psi(\cdot,\eta)$ (see \remref{remP}). 
	But first, we need to introduce a suitable parameter space for $\Phi$.
	\begin{definition}\label{defn3.2}
		Let $a_0>1/2$ be as in \eqref{3.6}.
		Let $0<\delta\ll1,\,0<b<1/2$ be as in \defnref{defn3}.
				Let $C>0$ be a large constant depending on
				the absolute, implicit constant in \eqref{5.01}.
		
		The space $$\Bc_\delta\subset X^s(b)\subset X^s\subset  X^s(a_0)$$ consists of all
		functions $\eta=\eta(y,\om)$ s.th.
		\begin{align}
			\label{e1} &\norm{\eta}_s<\delta,\\
			\label{e2}&\norm\eta_{s,b}^2\le c/2, \quad \text{where the number $c>0$ 
		is as in \eqref{3.3'}},\\
	\label{e3}
		&\text{$\eta(y,\om)\ge\delta$ and $\abs{\eta(y,\om)}\ge C\delta^2\abs{y}^2$
		for $\abs{y}^2\ge C^{-1}\delta^{-1}$.}
		\end{align}
		
		The space $\Sc=X^s(a_0)$ 
		is the orthogonal complement to all the zero-unstable modes in \eqref{2.4}-\eqref{2.8}  of 
		the linearized operator $L(a_0)$ w.r.t. the $X^0(a_0)$-inner product defined in \eqref{inn}.
	\end{definition}
	
	We will explain the the meaning of this parameter space below.
	For now, observe
	that the codimension
	of $\Sc$ is the sum of the multipliciteis of all the non-positive 
	eigenvalues of $L(a_0)$, which equals to 
	$$\text{codim}\, \Sc =n+2+\frac{(n-k)(n-k+3)}{2}.$$
	Notice that, as pointed out in \cite{MR3374960}*{Eqn. (3.44)}, not 
	all of the zero-unstable modes in \eqref{2.4}-\eqref{2.8} are linearly independent.
	
		Now we define the key map $\Phi$ in \thmref{thm1}.
	\begin{definition}[the quadratic correction $\Phi$]
		\label{Phi}
		For $\eta_0\in\Bc_\delta\cap \Sc$ as in \defnref{defn3.2},
		define 
		\begin{equation}
			\label{3.10}
			\Phi(\eta_0):= \beta_i(\eta_0) \Si^{(1,0)(i,0,0)}(a_0)+\g_{ij}(\eta_0)\Si^{(2,0),(i,j,0) }(a_0),
		\end{equation}
		where $\beta_i,\,\g_{ij}$ are defined in \thmref{thm5.1}.
	\end{definition}

\begin{remark}\label{remE}
		Consider the requirements from \defnref{defn3.2}.
	In view of the quadratic estimates \eqref{1.9.1},
	for sufficiently small $\delta$, conditions
	\eqref{e1}-\eqref{e2} ensure the compatibility 
	to impose the   initial condition
	$\xi(0)=\eta_0+\Phi(\eta_0)$ for 
	a path $\xi\in\Ac_\delta^\xi$. (I.e. this ensures $\norm{\xi(0)}_s\le\delta$ for
	sufficiently small $\delta$).
	Condition
	\eqref{e3} has to to with the geometric interpretation 
	of \thmref{thm1}.
	Indeed, for sufficiently small $\delta$, by the definition  \eqref{3.10} above
	and the formulae \eqref{2.6}, \eqref{2.8}, condition
	\eqref{e3} ensures the function $\eta_0+\Phi(\eta_0)\ge0$
	on the entire cylinder.
	In this case, we can interpret this function as a normal graph over the 
	the cylinder, parametrizing the hypersurface
	$$Y_0=(y,(\sqrt{k/a_0}+\eta_0+\Phi(\eta_0))\om),\quad y\in\Rb^{n-k},\,\om\in\Sb^k.$$
	By the avoidance principle for MCF, 
	it follows that a flow satisfying \eqref{1.6} 
	with initial configuration $Y_0$ above remains 
	to be a normal graph over the cylinder of radius $\sqrt{k/a_0}$ for all time.
	This justifies the geometric meaning for the analysis of  PDEs
	in terms of the graph function $\xi$
	in the remaining sections. 
\end{remark}

	Using the  modulation equations from \secref{sec:2.3}, we can ensure that
	a flow generated by a function $\eta_0\in \Bc_\delta\cap \Sc$ under
	\eqref{2.1'} remains orthogonal at all $\tau\ge0$ 
	to all of the zero-unstable modes
	that are due to broken symmetry (see a discussion at the beginning
	of \secref{sec:2.3}). 
	Yet, this does not suffice to give any dissipative
	estimate because of the remaining zero-unstable modes that 
	cannot be eliminated by the modulation equations, namely 
	$\Si^{(1,0)}$ (horizontal translation) and $\Si^{(2,0)}$ (shape instability).
	The former is protected by the symmetry of cylinder lying along the $y$-axis.
	The latter is not due to any symmetry.

\subsection{Proof of the Main Results}
\label{sec:3.2}
In this subsection we assume \thmref{thm4.1}-\thmref{thm6.1} hold. 
Then we prove the main results from \secref{sec:1}.

	\begin{proof}[Proof of \thmref{thm1}, assuming \thmref{thm4.1}-\thmref{thm6.1}]
		By construction, the fixed point $(\si,\xi)$ in $\Ac_\delta$ of the map $\Psi$
		solves the graphical rescaled MCF \eqref{2.1'}, coupled to the modulation
		equations \eqref{2.17}-\eqref{2.15}, 
		with initial configuration given by 
		$$\si(0)=(\delta_{ij},0,a_0),\quad \xi(0)=\eta_0+\Phi(\eta_0).$$
		This amounts to a global solution to the rescaled MCF \eqref{1.6} 
		satisfying
		\begin{equation}\label{3.11}
			\dot \si=\vec F(\si)\xi+\vec M(\si,\xi), \quad Y(t)=(y,(\sqrt{k/a(t)}+\xi(t))\om),
		\end{equation}
		together with the initial configuration
		\begin{equation}\label{3.12}
			\si(0)=(\delta_{ij},0,a_0),\quad Y(0)=(y,(\sqrt{k/a_0}+\eta_0+\Phi(\eta_0))\om).
		\end{equation}
		Estimates \eqref{1.9.1}-\eqref{1.9.2} are the content of 
		\thmref{thm6.1}.
		Thus Parts 1-2 of the theorem is proved. 
		The positivity of $\xi$ follows from the fact that 
		$\xi(0)\ge0$, which holds by \eqref{e3}, and the avoidance principle for MCF,
		c.f. \remref{remE}.
		The remainder estimates 
		\eqref{1.13}-\eqref{1.14} come from
		the decay condition in the space $\Ac_\delta$ in which
		$\Psi$ is a contraction, namely \eqref{3.1}-\eqref{3.2}.

			\end{proof}
		
		\begin{proof}[Proof of \corref{cor1}]
	Without loss of generality, suppose $t_0=0$.
	Let  $X$ be a maximal solution to \eqref{MCF},
	and suppose for some $\l_0>0$,
	the hypersurface  $\l_0^{-1}X(\cdot,0)\in M$ with $M$ given in \eqref{M}.
		Then, by \thmref{thm1}, $X$ is of the form \eqref{1.8'}
		for all $t<T$, and $X$ is uniquely determined
		by the pair $\si(\tau(t)),\xi(y(x,t),\om,\tau(t))$
		constructed in \thmref{thm5.1} below.

		In terms of the slow time variable $\tau$,
		the Cauchy problem for $\si$ in \eqref{1.11}-\eqref{1.12}
		uniquely determines the 
		axis 
		and radius of the limit cylinder given the initial profile. 
		This system of ODEs does not 
		depend on taking any particular  sequence of $\l\to\infty$
		in the rescaling procedure,
		and the solution to the Cauchy problem is unique by \thmref{thm5.1}. 
		Thus \corref{cor1} follows.
		\end{proof}

	\section{The mapping property of $\Psi$}\label{sec:4}
	In this section we prove that for appropriately chosen constants
	$\beta_i,\,\g_{ij}$ in \eqref{3.7}
	the $\Psi$ maps from $\Ac_\delta$ into itself.
	
	Recall that $\Psi$ maps a fixed path $(\si^{(0)},\xi^{(0)})\in \Ac_\delta$
	to the global solution to \eqref{3.8} with initial configuration \eqref{3.6}-\eqref{3.7},
	constructed explicitly in \lemref{lemA.1}.
	
	\begin{theorem}
		\label{thm4.1}
		For every $\eta_0\in\Bc_\delta\cap\Sc$ as in \defnref{defn3.2}
		and every fixed path $(\si^{(0)},\xi^{(0)})\in \Ac_\delta$, there exist
		unique coefficients $\beta_i,\,\g_{ij}$, depending on the choice of
		$\si^{(0)},\,\xi^{(0)}$ only,
		s.th. the solution to \eqref{3.4}-\eqref{3.7} lies in
		$\Ac_\delta$.
		
		Moreover, 
		 there hold
		the quadratic estimates
		\begin{align}
			\label{4.1}
			&\abs{\beta_i(\si^{(0)},\xi^{(0)})}\ls \delta^2,\\
			\label{4.1'}
			&\abs{\g_{ij}(\si^{(0)},\xi^{(0)})}\ls \delta^2.
		\end{align}
	
		Moreover, for two given paths $\si^{(m)},\,\xi^{(m)},m=1,2$, there hold
		the Lipschitz estimates
		\begin{align}
			\label{4.2}
			&\abs{\beta_i(\si^{(0)},\xi^{(0)})-\beta_i(\si^{(1)},\xi^{(1)})}\ls \delta\norm{(\si^{(0)},\xi^{(0)})-(\si^{(1)},\xi^{(1)}) },\\
			\label{4.2'}
			&\abs{\g_{ij}(\si^{(0)},\xi^{(0)})-\g_{ij}(\si^{(1)},\xi^{(1)})}\ls\delta \norm{(\si^{(0)},\xi^{(0)})-(\si^{(1)},\xi^{(1)}) },\\
		\end{align}
	where the norm in the r.h.s. is defined in \eqref{5.0}.

	\end{theorem}
	\begin{proof}
		1. 
		Let $(\si,\xi):=\Psi(\si^{(0)},\xi^{(0)})$. We want to check this pair
		$(\si,\xi)$ satisfies \eqref{3.1}-\eqref{3.3'}.
		In view of 
		the regularity result \lemref{lemA.2}, it remains to check the decay conditions
		\eqref{3.2} only. 
		
		Indeed,  condition \eqref{3.2}
		does not hold in
		general, 
		 because in principle the solution construct in
		\lemref{lemA.1} may not decay due to the 
		zero-unstable modes that are not eliminated by the modulation equations. 
		Once the decay condition \eqref{3.2} for $\xi$ is proved, we also get the 
		velocity bound \eqref{3.1} for $\si$, from the equation \eqref{3.5}.
		
		2. Let $P^{(m,n)} (\tau)$ be the projection onto the span of the zero-unstable modes $\spn\Set{\Si^{(m,n)(i,j,l)}(a^{(0)}(\tau))}$ of $L(a^{(0)}(\tau))$,
		and write $\xi^{(m,n)}:=P^{(m,n)}\xi$.
		Let $Q:=1-\sum P^{(m,n)}$ and write $\xi_S=Q\xi=\xi-\sum \xi^{(m,n)}$.
		Here the superscript $S$ refers to  stable projection.
		
		By the orthogonality condition \eqref{A.0},
		we know $\xi^{(m,n)}=0$ for $(m,n)=(0,0),(0,1),(1,1)$.
		Now we expand 
		$$\xi=\xi_S+\sum_{m=1,2}\xi^{(m,0)},$$
		and plug this expansion into \eqref{3.8}.
		Using the orthogonality among various eigenfunctions of 
		the self-adjoint operator $L(a^{(0)}(\tau))$
		in the space $X^s(a^{(0)}(\tau))$ (notice that there
		is no essential spectrum because of the compact $\Sb^k$ factor
		in the domain), 
		we find that equation \eqref{3.8} is equivalent to the following system:
		\begin{align}
			\label{4.3}
			&\dot \xi_S-QL(a^{(0)})\xi_S=-Q N(a^{(0)},\xi^{(0)}),\\
			\label{4.4}
			&\dot \beta_i-a^{(0)}\beta_i=f_i,\\
			\label{4.5}
			&\dot \g_{ij}=h_{ij},
		\end{align}
		where
		\begin{align}
			\label{4.6}
			&f_i=-
			\inn{ N(a^{(0)},\xi^{(0)})}{\Si^{(1,0)(i,0,0)}(a^{(0)})}_{a^{(0)}},\\
		&\label{4.7}
		h_{ij}=-
\inn{ N(a^{(0)},\xi^{(0)})}{\Si^{(2,0)(i,j,0)}(a^{(0)})}_{a^{(0)}}.
	\end{align}
	(Notice that $\Si^{(1,0)}$ and $\Si^{(2,0)}$ are already normalized.)
	To get \eqref{4.6}-\eqref{4.7}, which are the projections of the
	nonlinearity into the eigenspaces of $P^{(1,0)}$ and $P^{(2,0)}$ respectively,
	we use the fact that since $\di_\si W(\si^{(0)})$ maps into $\ran P^{(0,0)}\oplus\ran P^{(0,1)}\oplus \ran P^{(1,1)}$, the terms in \eqref{3.8} involving $\W=\di_\si W\vec F$ and $\di_\si W\vec M$ all  drop out
	after taking projection into the other eigenspaces.
	
	The initial configurations associated to \eqref{4.3}-\eqref{4.5}
	are the three terms in \eqref{3.7}, respectively. 
	Since \eqref{4.3}-\eqref{4.5} are already decoupled,
		in what follows  we consider the Cauchy problem for $\xi_S,\,\beta_i,\,\g_{ij}$ separately.

		First, for \eqref{4.3}, consider the following representation formula
	of the solution $\xi_S(\tau)$ by Duhamel's principle:
	\begin{equation}\label{4.3'}
		\xi_S(\tau)=e^{-\tau \b L }\xi_S(0)-\int_0^\tau e^{-(\tau-\tau')\b L}Q N(a^{(0)}(\tau'),\xi^{(0)}(\tau'))\,d\tau'.
	\end{equation}
	Here $\b L:= QL(a^{(0)})Q$.
	Since at $\tau=0$, the function $\xi_S(0)=\eta_0$ is perpendicular to all
	the zero-unstable modes of $L(a^{(0)}(0))$ in the space $X^s(a_0)$,
	and for all $\tau>0$ the operator $\b L$ is uniformly coercive on $X^s(a^{(0)}(\tau))$ (which follows from Item 2 in \defnref{defn2.1}), 
	it follows from standard parabolic theory (see \secref{sec:A} for details)
	that there exists  $c>0$ 
	depending on $n,k$ only s.th.
	\begin{equation}
		\label{4.8'}
		\norm{\xi_S(\tau)}_{s,a(\tau)}\le e^{-c\tau}\norm{\xi_S(0)}_{s-2,a(\tau)},\quad  \tau\ge0.
	\end{equation}
	By the interpolation inequality in \lemref{lemB.0},
	it follows 
		\begin{equation}
		\label{4.8}
		\norm{\xi_S(\tau)}_{s}\le \delta e^{-c'\tau},\quad  \tau\ge0,
	\end{equation}
	where $c'$ depends on the constant $c$ in \eqref{4.8'} only,
	assuming $\beta_i,\,\g_{ij}=O(\delta^2)$ in \eqref{3.7}.
	
	Next, for \eqref{4.4},  by \lemref{lemB.1}, we find that
	$\beta_i(\tau)$ remains bounded for all $\tau$ if and only if the initial configuration is given by
	\begin{equation}
		\label{4.9}
		\beta_i(0)=-\int_{0}^\infty f_i(\tau')\l(\tau')^{-1}\,d\tau'\quad \del{\l(\tau)=e^{\int_0^\tau a(\tau')}\,d\tau'}.
	\end{equation}
	This integral indeed converges, since $\l(\tau)\ge1$, and the nonlinear estimate \eqref{A.3}
	shows $\abs{f_i(\tau)}\le\norm{P^{(1,0)}N}_{0,a^{(0)}}\ls \norm{\xi^{(0)}(\tau)}_s^2\ls \br{\tau}^{-4}.$
	For this estimate, we use the fact that $a^{(0)}$ is admissible, c.f. \defnref{defn2.1}, together with the decay condition \eqref{3.2}.
	
	\eqref{4.9} determines the choice of $\beta_i$ in \eqref{3.7}.	
	Since $f_i$ is a functional of $\si^{(0)},\xi^{(0)}$, this choice of $\beta_i$ 
	is a function of the fixed path $\si^{(0)},\xi^{(0)}$.

	Suppose \eqref{4.9} is satisfied. Then $\beta_i(\tau)$ is given by
	\begin{equation}\label{4.10}
		\beta_i(\tau)=-\int_\tau^\infty f(\tau')\frac{\l(\tau)}{\l(\tau')}\,d\tau'.
	\end{equation}
	See \lemref{lemB.1} for details. For all $\tau\ge0$, this function satisfies 
	\begin{equation}
		\label{4.11'}
		\begin{aligned}
			\abs{\beta_i(\tau)}&\le \int_\tau^\infty \abs{f_i(\tau')}e^{-\int_\tau^{\tau'}a}\,d\tau'\\&\le \int_\tau^\infty \abs{f_i(\tau')}d\tau'\\
			&\le C \delta^2\br{\tau}^{-3}.
		\end{aligned}
	\end{equation}
	Here the last inequality follows from \eqref{3.2} and the estimate
	$\abs{f_i}\le \norm{P^{(1,0)}N}_{0,a^{(0)}}\ls \norm{\xi^{(0)}}_s^2\le \delta^2\br{\tau}^{-4},$ due to the nonlinear estimate \eqref{A.3}. For the latter, 
	we need the lower bound from Item 3 in the admissibility conditions
	listed in  \defnref{defn2.1}, which holds for $\si^{(0)}\in \Ac_\delta^\si$. 
	In \eqref{4.11'}, the
	 constant $C$ depends only on the  constant in the estimate \eqref{A.3},
	which in turn depends only on the dimension $n$. 
	
		We conclude from \eqref{4.11'} that  
	\begin{equation}
		\label{4.12}
		\abs{\beta_i(0)}\ls \delta^2.
	\end{equation}
	This gives \eqref{4.1}.
	
	By \eqref{4.11'}, the elliptic estimate \eqref{A.5}, and the interpolation
	inequality \eqref{B.8},  we conclude
	\begin{equation}
		\label{4.11}
		\norm{P^{(1,0)}\xi(\tau)}_s \le \delta\br{\tau}^{-2},
	\end{equation}
	provided $\delta$ is sufficiently small,
	in the sense that $\delta^{1/3}Cc^{1/12}c_1\le1$ for the constants
	 $C$ from \eqref{4.11'},
		$c$ the pivot condition \eqref{B.7}, and $c_1$ from some elliptic
		estimate, respectively.
	Notice that the path $\xi$ satisfies the pivot condition \eqref{B.7},
	as we show in \lemref{lemA.2}. This justifies the use of the interpolation
	\eqref{A.5}.

	Lastly, for \eqref{4.5}, we note that $\g_{ij}(\tau)\to 0$ as $\tau\to 0$ if and only if
	\begin{equation}
		\label{4.13}\g_{ij}(0)=-\int_0^\infty h_{ij}.
	\end{equation}
	This defines the choice of $\g_{ij}$ in \eqref{3.7}. 
	Moreover, if \eqref{4.13} holds, then 
	\begin{equation}
		\label{4.14}
		\g_{ij}(\tau)=-\int_\tau^\infty h_{ij},
	\end{equation}
	and there holds
	\begin{equation}
		\label{4.15'}
		\abs{\g_{ij}(\tau)}\le C' \delta^2\br{\tau}^{-3}.
	\end{equation}
	Here the constant $C'$ is a constant depending on the estimate
	\eqref{A.3}, similar to the one from \eqref{4.11'} (but possibly larger).
	From here   we conclude \eqref{4.1'} by setting $\tau=0$.
	Shrinking $\delta$ if necessary, we can also deduce from \eqref{4.15'} the following estimate, whose derivation is identical to \eqref{4.11}:
	\begin{equation}
		\label{4.15}
		\norm{P^{(2,0)}\xi(\tau)}_s \le \delta \br{\tau}^{-2}.
	\end{equation}

	Combining \eqref{4.8}, \eqref{4.11}, \eqref{4.15},
	we conclude that with the unique choices \eqref{4.11}, \eqref{4.14}
	of $\beta_i,\,\g_{ij}$,
	as functions of $\xi^{(0)},\si^{(0)}$,
	the solution $\xi$ to \eqref{3.8} with initial condition
	\eqref{3.7}
	satisfies the decay condition \eqref{3.2}. For reason described at the beginning
	of the proof, this implies \eqref{3.1} holds for $\si$.
	This proves the 	 the image $(\si,\xi)=\Psi(\si^{(0)},\xi^{(0)})$ lies
	in $\Ac_\delta$.

	3. Now we prove the estimates \eqref{4.2}-\eqref{4.2'}.
	For simplicity we only prove \eqref{4.2} because 
	the other is completely analogous. 
	
	Write  $U^m:=(\si^{(m)},\xi^{(m)}),\,m=0,1$ for two paths in $\Ac_\delta$,
	and $\beta^{(m)}_i=\beta_i(U^m)$ for the constants
	in the initial configurations associated to these paths, as constructed above.
	
	By \eqref{4.9}, we find 
	\begin{equation}\label{4.16}
		\abs{\beta_i^{(1)}-\beta_i^{(0)}}\le \int_0^\infty \abs{f_i^{(1)}(\tau')-f_i^{(0)}(\tau')}\,d\tau',
	\end{equation}
	where $f_i^{(m)}$ is given by \eqref{4.6}.
	\begin{equation}
		\label{4.17}
		f_i^{(m)}=l_i^{(m)} N(U^m),
	\end{equation}
	where $ N$ is as in \eqref{3.4}, and $l_i^{(m)}:X^{s-2}\to \Rb$ is the linear
	functional  given by
	$\phi\mapsto \inn{\phi}{\Si^{(1,0)(i,0,0)}(a^{(m)})}$.

	Consider the difference of the r.h.s. of \eqref{4.17}, which we write as
	\begin{equation}
		\label{4.18}
		l_i^{(1)} N(U^1)-l_i^{(0)} N(U^0)=(l_i^{(1)}-l_i^{(0)}) N(U^1)+l_i^{(0)}( N(U^1)- N(U^0)).
	\end{equation}
		Since the paths $\si^{(m)}$ are admissible (see \defnref{defn2.1}), 
		we have the uniform estiamte $\norm{l_i^{(m)}}_{X^{s-2}\to \Rb}\ls \abs{a^{(m)}}$
	independent of $m,\,i$ and $\tau$.  Using this, we can bound the first term in the r.h.s. of \eqref{4.18} by $\abs{\si^{(1)}-\si^{(0)}}\norm{ N(U^1)}_{s-2}$.
	By \eqref{A.3} and \eqref{3.2}, this implies 
		\begin{equation}\label{4.19}
			\abs{(l_i^{(1)}-l_i^{(0)}) N(U^1)}\ls
			 \delta^2\br{\tau}^{-4}\abs{\si^{(1)}-\si^{(0)}}.
		\end{equation}
	By the Lipschitz estimate  \eqref{5.9} (see \secref{sec:C} for the proof) and \eqref{3.1}, the second term in the r.h.s.
	of \eqref{4.18} 
	can be bounded as
			\begin{equation}\label{4.20}
\begin{aligned}
	\abs{l_i^{(0)}( N(U^1)- N(U^0))}&\ls \norm{l_i^{(m)}}_{X^{s-2}\to \Rb} \norm{ N(U^1)- N(U^0)}_{s-2}\\&\ls \delta\br{\tau}^{-2} \del{\norm{\xi^{(1)}-\xi^{(0)}}_s+\abs{\si^{(1)}-\si^{(0)}}}.
\end{aligned}
	\end{equation}
Combining \eqref{4.19}-\eqref{4.20} in \eqref{4.18}, and plugging the result back to \eqref{4.16},
we find 
$$\abs{\beta_i^{(1)}-\beta_i^{(0)}}\ls  \delta\del{\int \br{\tau'}^{-2}\,d\tau'}\norm{U^1-U^0}.$$
This gives  \eqref{4.2}. The same argument gives \eqref{4.2'}.

	\end{proof}

	\section{The contraction property of $\Psi$}\label{sec:5}
	In this section we prove the map $\Psi:\Ac_\delta\to \Ac_\delta$
	defines a contraction w.r.t. the following norm:
	\begin{equation}
		\label{5.0}
		\norm{(\si,\xi)}=\sup_{\tau\ge0}(c_0^{-1}\br{\tau}\abs{\si(\tau)-\tau(0)}+\br{\tau}^{2}\norm{\xi(\tau)}_s).
	\end{equation}
	Here $c_0>0$ is the constant in \eqref{3.1}.
	This defines a  norm equivalent to the uniform one in the space 
	$C([0,\infty),\Si\times X^s)$,
	in which the 
	the space $\Ac_\delta$ from \defnref{defn3} 
	is the closed ball of  size $O(\delta)$ around zero
	w.r.t. the norm \eqref{5.0}.
	This shows $(\Ac_\delta,\norm{\cdot})$ is complete.
	
	Now we show $\Psi$ is a contraction in  $(\Ac_\delta,\norm{\cdot})$.
	This essentially implies all of the statements in \thmref{thm1}.

\begin{theorem}
	\label{thm5.1}
	The map $\Psi:\Ac_\delta\to \Ac_\delta$ defined in \eqref{3.9} is a contraction with respect to the norm \eqref{5.0}, satisfying
	\begin{equation}\label{5.00}
		\norm{\Psi(U^1)-\Psi(U^2)}\le \delta^{1/2} \norm{U^1-U^0}\quad (U^1,\,U^2\in \Ac_\delta).
	\end{equation}
	
	Consequently, there exist unique constants 
	\begin{equation}\label{5.01}
		\beta_i(\eta_0),\,\g_{ij}(\eta_0)=O(\delta^2),
	\end{equation}
s.th.
	there exists a unique solution to the nonlinear system \eqref{2.1'}, \eqref{2.17}-\eqref{2.15} with initial condition
	\eqref{3.6} and this choice of $\beta_i$, $\g_{ij}$ in \eqref{3.7}.

\end{theorem}
\begin{remark}
	As a by-product, this proves a global well-posedness result for the rescaled MCF \eqref{1.6}. 
		The admissible set of initial configurations
	forms a finite-codimensional manifold in $X^s$.
	The global solutions constructed here does not 
	necessarily arise from rescaling a given solution to \eqref{MCF}.
\end{remark}
\begin{proof}
	1. For simplicity, write $U^m:=(\si^{(m)},\xi^{(m)}),\,m=0,1$ for two elements in $\Ac_\delta$,
	and $U^{m+2}:=\Psi(\si^{(m)},\xi^{(m)})$.
	
	Using \eqref{3.4}-\eqref{3.7}, we find that the difference $U^3-U^2$ is the unique 
	solution to the following linear
	Cauchy problem:
	\begin{align}
\label{5.3}\begin{split}
			\di_\tau(\xi^{(3)}-\xi^{(2)})+L(a^{(0)})(\xi^{(3)}-\xi^{(2)})=& V(\si^{(0)},\si^{(1)})\xi^{(3)}\\&-(\dot \si^{(3)}-\dot\si^{(2)})\di_\si W (\si^{(0)})\\&-\dot\si^{(3)}(\di_\si W(\si^{(1)})-\di_\si W(\si^{(0)}))\\&-N(a^{(1)},\xi^{(1)})+N(a^{(0)},\xi^{(0)}),
\end{split}\\
\label{5.4}
\begin{split}
	(\xi^{(3)}-\xi^{(2)})\vert_{\tau=0}=&(\beta^{(1)}_i-\beta^{(0)}_i)\Si^{(1,0)(i,0,0)}(a_0)
\\&+(\g_{ij}^{(1)}-\g_{ij}^{(0)})\Si^{(2,0)(i,j,0)}(a_0).
\end{split}
	\end{align}
In \eqref{5.3}, the operator $V$ is defined as the difference between
the two linearized operators as $V:=L(a^{(1)})-L(a^{(0)})$. 
In \eqref{5.4}, the numbers $\beta_i^{(m)},\,\g_{ij}^{(m)},\,m=0,1$
are the unique functions of $U^m$ constructed in \thmref{thm1}.
To get \eqref{5.3}-\eqref{5.4}, as before, we use the equivalence between the system \eqref{3.4}-\eqref{3.7} in $U^m$
and the single equation in $\xi^{(m)}$, i.e. \eqref{3.8} with \eqref{3.7}.

2. The
claim now is that we have the following point-wise  estimate for the solution to \eqref{5.3}-\eqref{5.4}:
\begin{equation}	\label{5.5}
	\norm{\xi^{(3)}(\tau)-\xi^{(2)}(\tau)}_s\ls\delta\br{\tau}^{-2}\norm{U^1-U^2},\quad \tau\ge0.
\end{equation}

To derive \eqref{5.5}, we consider the projections of $\xi^{(3)}-\xi^{(2)}$ under $P^{(m,n)}_\mu(\tau):=P^{(m,n)}(a^{(\mu)}(\tau)),\,\mu=1,2$
 and $Q_\mu=1- \sum P^{(m,n)}_i$.
These operators are defined in the proof of \thmref{thm4.1}.

For $(m,n)=(0,0),(0,1),(1,1)$, it suffices to use the orthogonality condition \eqref{A.0}.
Indeed, since $P_i^{(m,n)}\xi^{\mu+2}=0$ for $\mu=1,2$, we find
we have
\begin{equation}\label{5.6}
	\begin{aligned}
		\norm{P_0^{(m,n)}(\xi^{(3)}-\xi^{(2)})}_s&=\norm{P_0^{(m,n)}\xi^{(3)}}_s\\
		&=\norm{(P^{(m,n)}_1-P^{(m,n)}_0)\xi^{(3)}}_s\\
		&\ls \delta\abs{\si^{(1)}-\si^{(0)}}\br{\tau}^{-2}.
	\end{aligned}
\end{equation}
In the last step we use the fact that $\norm{P_1-P_2}_{s\to s}\ls \abs{\si^{(1)}-\si^{(0)}}$,
together with the condition \eqref{3.2} for $\xi^{(3)}$.

To bound the other projections of $\xi^{(3)}-\xi^{(2)}$, we need to use the equation \eqref{5.3}.
First, we bound each of the terms in the r.h.s. of \eqref{5.3}.
On a bounded domain in $X^s,\,s\ge2$, by the $C^2$-regularity of the $F$-functional,
the perturbation operator $V(\si^{(0)},\si^{(1)})$ satisfies 
the Lipschitz estimate $\norm{V}_{s\to s-2}\ls \abs{\si^{(1)}-\si^{(0)}}$. 
(Alternative, one can get this using \eqref{A.8}.)
This, together with the estimate  \eqref{3.2} for $\xi^{(3)}$, implies that the first term in the r.h.s. of \eqref{5.3} satisfies
\begin{equation}
	\label{5.7}
	\norm{V(\si^{(0)},\si^{(1)})\xi^{(3)}}_{s-2}\ls \delta \abs{\si^{(1)}-\si^{(0)}}\br{\tau}^{-2}.
\end{equation}

The second term in the r.h.s. of \eqref{5.3} vanishes in the equation for 
$P_0^{(m,n)}(\xi^{(3)}-\xi^{(2)})$ with $(m,n)\ne(0,0),(0,1),(1,1)$ and $Q_0(\xi^{(3)}-\xi^{(2)})$,
because the operator $\di_\si W(\si^{(0)})$ maps into the space $\ran P_0^{(0,0)}\oplus\ran P^{(0,1)}_0\oplus \ran P^{(1,1)}_0$. Thus we do not need to consider this term.

As discussed in \remref{remW} earlier, the operator 
$\di_\si W$ is locally Lipschitz from the tangent bundle
$T\Si\to X^s$. This, together with the estimate \eqref{3.1} for 
$\abs{\dot\si^{(3)}}$ , implies  that the third term in the r.h.s. of \eqref{5.3} 
satisfies 
\begin{equation}
	\label{5.8}
	\norm{\dot\si^{(3)}(\di_\si W(\si^{(1)})-\di_\si W(\si^{(0)}))}_s\ls\delta \abs{\si^{(1)}-\si^{(0)}}\br{\tau}^{-2}.
\end{equation}

Lastly, in Appendix \ref{sec:C}, we show the estimate
\begin{equation}
	\label{5.9}
	\norm{N(a^{(1)},\xi^{(1)})-N(a^{(0)},\xi^{(0)})}_{s-2}\ls\delta\br{\tau}^{-2} \del{\norm{\xi^{(1)}-\xi^{(0)}}_s+\abs{\si^{(1)}-\si^{(0)}}},
\end{equation}
c.f. \eqref{C.2}. So far, this has the weakest decay property.

3. We conclude from \eqref{5.8}-\eqref{5.9} that
\begin{equation}
	\label{5.10}
	\norm{\text{r.h.s. of \eqref{5.3}}}_{s-2}\ls \delta\br{\tau}^{-2} \left(\norm{\xi^{(1)}-\xi^{(0)}}_s+ \abs{\si^{(1)}-\si^{(0)}}\right).
\end{equation}
A similar estimate holds if we project both sides of \eqref{5.3} with
$Q_0$ and $P_0^{m,0},\,m=1,2$. 

Below we will use \eqref{5.10} as follows: Suppose $P$ is one of the projections
$Q_0,\,P_0^{m,0},\,m=1,2$. 
Consider the integral over the r.h.s. of \eqref{5.10},
\begin{equation}\label{5.10.1}
	\int_0^\tau \br{\tau'}^{-2}\del{\norm{P(\xi^{(1)}(\tau')-\xi^{(0)}(\tau'))}_s+ \abs{\si^{(1)}(\tau')-\si^{(0)}(\tau')} }\,d\tau'.
\end{equation}
The integrand is bounded by 
$$\br{\tau}^{-3}\sup_{0\le \tau' \le \tau}\del{\br{\tau'}\norm{P(\xi^{(1)}(\tau')-\xi^{(0)}(\tau'))}_s+\br{\tau'} \abs{\si^{(1)}(\tau')-\si^{(0)}(\tau')}}.$$
Plugging this to \eqref{5.10.1}, and factoring the constant out, 
we find the estimate 
\begin{multline}\label{5.10.2}
		\int_0^\tau \br{\tau'}^{-2}\del{\norm{P(\xi^{(1)}(\tau')-\xi^{(0)}(\tau'))}_s+ \abs{\si^{(1)}(\tau')-\si^{(0)}(\tau')} }\,d\tau'\\
		\ls
	\br{\tau}^{-2}\norm{(P(\xi^{(1)}-\xi^{(0)}),\si^{(1)}-\si^{(0)})},
\end{multline}
where the norm on the r.h.s. is defined in \eqref{5.0}. The same argument gives 
the estimate 
\begin{multline}\label{5.10.3}
	\int_\tau^\infty \br{\tau'}^{-2}\del{\norm{P(\xi^{(1)}(\tau')-\xi^{(0)}(\tau'))}_s+ \abs{\si^{(1)}(\tau')-\si^{(0)}(\tau')} }\,d\tau'\\
	\ls
	\br{\tau}^{-2}\norm{(P(\xi^{(1)}-\xi^{(0)}),\si^{(1)}-\si^{(0)})}.
\end{multline}

4. Now consider the equation for $Q_0(\xi^{(3)}-\xi^{(2)})$, which can be derived from \eqref{5.3}
as \eqref{4.3} from \eqref{3.8}. Applying Duhamel's principle to this equation,
we find
\begin{equation}
	\label{5.11}
\begin{aligned}
		\norm{Q_0(\xi^{(3)}-\xi^{(2)})}_s \le& \norm{e^{-\tau \b L}(Q_0(\xi^{(3)}-\xi^{(2)})(0))}_s\\&+\delta\int_0^\tau \norm{e^{-(\tau -\tau')\b L}}_{s-2,s}\br{\tau'}^{-2}\Bigl(\norm{Q_0(\xi^{(1)}(\tau')-\xi^{(0)}(\tau'))}_{s-2}\\&+\abs{\si^{(1)}(\tau')-\si^{(0)}(\tau')}\Bigr)\,d\tau',
\end{aligned}
\end{equation}
where $\b L:= Q_0L(a^{(0)})Q_0: X^s\to X^{s-2}$, and $e^{-\tau \b L}: X^{s-2}\to X^s$ is the propagator of $\b L$.

We consider the two terms in the r.h.s. of \eqref{5.11}.
Since the stable projection of  the initial configuration \eqref{5.4}
vanishes,  the first term in the r.h.s. of \eqref{5.11}
drops out.

Since $Q_0$ projects onto the stable modes of $L$, 
the restriction $\b L$ is uniformly coercive, with $\inn{\b L\phi }{\phi}_{a^{(0)}}\ge c \norm{\phi}_{0, a^{(0)}}^2$ for some $c>0$ depending on the decay property \eqref{3.1} only and any $\phi\in X^s(a^{(0)}),\,s\ge2$. 
This and the interpolation from  \lemref{lemB.0} gives the propagator estimate $\norm{e^{-\tau \b L}}_{s-2\to s} \ls 1$.
Hence, the second term in \eqref{5.11} reduces to an intergral of the 
form \eqref{5.10.1}. By \eqref{5.10.2}, we conclude
\begin{equation}\label{5.12}
	\norm{Q_0(\xi^{(3)}-\xi^{(2)})(\tau)}_s \ls \br{\tau}^{-2} \norm{(Q_0(\xi^{(1)}-\xi^{(0)}),\si^{(1)}-\si^{(0)})},\quad \tau\ge0.
\end{equation}

Next, consider the term $P_0^{(1,0)}(\xi^{(3)}-\xi^{(2)})$. 
As in the proof of \thmref{thm4.1}, 
the evolution
of this part amounts to the following system of ODEs:
\begin{equation}
	\label{5.13}
	\dot {\tilde\beta}_i-a^{(0)}\tilde{\beta_i}=\tilde{f_i},\quad i=1,\ldots,n-k
\end{equation}
where
\begin{equation}
	\label{5.14}
\begin{aligned}
		\tilde{f_i}=&-
	\inn{V(\si^{(0)},\si^{(1)})\xi^{(3)}}{\Si^{(1,0)(i,0,0)}(a^{(0)})}_{a^{(0)}}
	\\&-\inn{\dot\si^{(3)}(\di_\si W(\si^{(1)})-\di_\si W(\si^{(0)}))}{\Si^{(1,0)(i,0,0)}(a^{(0)})}_{a^{(0)}}
	\\&-\inn{N(a^{(1)},\xi^{(1)})-N(a^{(0)},\xi^{(0)})}{\Si^{(1,0)(i,0,0)}(a^{(0)})}_{a^{(0)}},
\end{aligned}
\end{equation}
and the initial configuration is $\beta_i^{(1)}-\beta_i^{(0)}$, c.f. \eqref{5.4}.

Since $\xi^{(3)}-\xi^{(2)}$ remains bounded for all time, so do the coefficients
$\tilde{\beta_i}$ in \eqref{5.13}. Thus \lemref{lemB.1} implies 
\begin{equation}
	\label{5.15}
	\begin{aligned}
		&\norm{P_0^{(1,0)}(\xi^{(3)}-\xi^{(2)})}_s\\
		\ls& \sum_{i=1}^{n-k} \abs{\tilde\beta_i}
		\le\sum_{i=1}^{n-k} \int_\tau^\infty \abs{\tilde f_i(\tau')}\frac{\lambda(\tau)}{\lambda(\tau')}\,d\tau'\\
		\ls&\delta\sum_{i=1}^{n-k}\int_\tau^\infty \br{\tau'}^{-2}\Bigl(\norm{P^{(1,0)(i,0,0)}(\xi^{(1)}(\tau')-\xi^{(0)}(\tau'))}_s\\&+ \abs{\si^{(1)}(\tau')-\si^{(0)}(\tau')}\Bigr)e^{-(\tau'-\tau)/2}\,d\tau'.
	\end{aligned}
\end{equation}
In the first step we use an elliptic estimate of the form \eqref{A.5}.
In the last step we use \eqref{5.10} and the formula \eqref{5.14}. Since
\eqref{5.15} can be bounded by an integral as in the l.h.s. of \eqref{5.10.3},
we conclude
\begin{equation}
	\label{5.16}
		\norm{P_0^{(1,0)}(\xi^{(3)}-\xi^{(2)})}_s\\ \ls \delta\br{\tau}^{-2}
		\norm{(P_0^{(1,0)}(\xi^{(1)}-\xi^{(0)}),\si^{(1)}-\si^{(0)})},
\end{equation}
where the norm on the r.h.s. is \eqref{5.0}.

The argument for estimating $P^{(2,0)}_0(\xi^{(3)}-\xi^{(2)})$ is completely analogous.
 In this case the implicit constant may be larger 
because we do not
have a factor as $e^{-\int_\tau^{\tau'}a}$ inside the integral that defines $\g_{ij}$
(see \eqref{4.14} and \eqref{B.4}). We record this estimate:
\begin{equation}
	\label{5.17}
		\norm{P_0^{(2,0)}(\xi^{(3)}-\xi^{(2)})}_s \ls \delta\br{\tau}^{-2}
\norm{(P_0^{(2,0)}(\xi^{(1)}-\xi^{(0)}),\si^{(1)}-\si^{(0)})}.
\end{equation}

Collecting \eqref{5.6}, \eqref{5.12}, \eqref{5.16}-\eqref{5.17} gives \eqref{5.5}.

5. It remains now to prove a similar estimate for $\si^{(3)}-\si^{(2)}$ as \eqref{5.5}. To
this end, we use the fact that this difference satisfies the Cauchy problem
\begin{align}\label{5.18}
	\begin{split}
	\di_\tau (\si^{(3)}-\si^{(2)})=&\vec F(\si^{(1)})(\xi^{(3)}-\xi^{(2)})\\&+ (\vec F(\si^{(1)})-\vec F(\si^{(0)}))\xi^{(2)}+\vec M(\si^{(1)},\xi^{(1)})-\vec M(\si^{(0)},\xi^{(0)}),
\end{split}\\
	\label{5.19}
	\begin{split}
		\si^{(3)}-\si^{(2)}\vert_{\tau=0}=&(0,0,0)\in \Si.
	\end{split}
\end{align}

By the definition of $\vec F$ in \secref{sec:3.1}, using Cauchy-Schwartz
we find  that the first term in the r.h.s. of \eqref{5.18} satisfies 
	\begin{equation}\label{5.19.1}
\begin{aligned}
			&\abs{\vec F(\si^{(1)})(\xi^{(3)}-\xi^{(2)})}\\\ls&\del{\abs{\si^{(1)}}+\abs{\dot \si^{(1)}} }\norm{\xi^{(3)}-\xi^{(2)}}_0\\\ls& \delta^2 \br{\tau}^{-3}\norm{U^1-U^2},
\end{aligned}
	\end{equation}
For the last inequality we use \eqref{5.5}.

By the definition of $\vec F$ and $\vec N$ (see \eqref{3.5}),
the two terms in the second line of \eqref{5.18} satisfy similar bound as
\eqref{5.8}-\eqref{5.9}, respectively.
Indeed, we find 
\begin{align}
	\label{5.19.2}
	\abs{\vec F(\si^{(1)})-\vec F(\si^{(0)}))\xi^{(2)}}&\ls \delta\br{\tau}^{-2}\abs{\si^{(1)}-\si^{(0)}}\br{\tau}^{-2},\\	
	\label{5.19.3}
	\abs{\vec M(\si^{(1)},\xi^{(1)})-\vec M(\si^{(0)},\xi^{(0)})}&\ls \delta\br{\tau}^{-2} \del{\norm{\xi^{(1)}-\xi^{(0)}}_s+\abs{\si^{(1)}-\si^{(0)}}}.
\end{align}

 We conclude from \eqref{5.19.1}-\eqref{5.19.3} that
\begin{equation}
	\label{5.20}
	\abs{\di_\tau (\si^{(3)}-\si^{(2)})} \ls \delta\br{\tau}^{-2} \left(\norm{\xi^{(1)}(\tau)-\xi^{(0)}(\tau)}_s+\abs{\si^{(1)}(\tau)-\si^{(0)}(\tau)}\right).
\end{equation}
Integrating \eqref{5.20} in time
and using the zero initial condition \eqref{5.19}, we find
the uniform bound
\begin{equation}
	\label{5.20'}
	\abs{\si^{(3)}(\tau)-\si^{(2)}(\tau)}\ls \br{\tau}^{-1} \norm{U^1-U^0}.
\end{equation}
This, together with \eqref{5.5}, gives the desired contraction property after taking supremum over $\tau$, namely
\begin{equation}
	\label{5.21}
	\norm{U^3-U^2}\le \delta^{1/2}\norm{U^1-U^0},
\end{equation}
provided $\delta^{1/2}$ is sufficiently small,
so as to cancel the implicit constants in \eqref{5.5} and \eqref{5.20'}.

6. By \thmref{thm4.1}, \eqref{5.21} and \lemref{lemB.2},
for every $\eta_0$ satisfying the assumption of \lemref{lemA.1},
we get a unique fixed point 
$U\in \Ac_\delta$ by iterating $\Psi$ on a point $U^{(0)}\in \Ac_\delta$.
Indeed, define $U^{(m)}:=\Psi(U^{m-1}),\,m=1,2,\ldots$.
Each $U^{(m)}$ is the unique solution 
to the linear Cauchy problem  \eqref{3.4}-\eqref{3.7}
with a initial configuration depending on 
$\eta_0$ and some constants $\beta_i^{(m-1)},\,\g_{ij}^{(m-1)}$.
By \thmref{thm4.1}, the latter are defined uniquely as functions of 
$U^{(m-1)}$.
It follows that the limit 
$U$ is the unique solution to the nonlinear problem 
\eqref{2.1'}, \eqref{2.17}-\eqref{2.15} with initial condition
\eqref{3.6}-\eqref{3.7}, where the constants $\beta_i,\,\g_{ij}$ 
are defined to be the limits of $\beta_i^{(m)},\,\g_{ij}^{(m)}$ as $m\to\infty$, respectively. Indeed, if $U^{(m)}$ is a Cauchy sequence in $\Ac_\delta$, then
these limits exist by \eqref{4.2}-\eqref{4.2'}.
Moreover, the constant  $\beta_i,\,\g_{ij}$ now depend on $\eta_0$ only.
By the uniform estimates \eqref{4.1}-\eqref{4.1'}, we conclude that
$\beta_i,\,\g_{ij}=O(\delta^2)$. This completes the proof.

\end{proof}

	\section{A Sable Manifold Theorem for the Rescaled MCF}\label{sec:6}
In this section we  consider \thmref{thm1} as 
a stable manifold theorem for the rescaled MCF, \eqref{2.1'}.

Define a set
\begin{equation}\label{6.0}
	M:=\Set{\eta+\Phi(\eta):\eta\in\Bc_\delta\cap \Sc}\quad (0<\delta\ll1).
\end{equation}
Recall that  the map $\Phi$, given in \defnref{Phi},
maps each  $\eta\in\Bc_\delta\cap \Sc$ in the parameter space
to the fixed point of the map
$\Psi(\cdot,\eta)$ in the space $\Ac_\delta$.

In what follows, we show 
the effect of $\Phi$ amounts to a second order perturbation of
$\Bc_\delta\cap \Sc$, and consequently the manifold
$M$ is non-degenerate.
We conjecture that in fact $\Phi$ is a $C^1$ immersion.  

We also show a Lipschitz estimate for $\Phi$,
This shows  $M$ is a Lipschitz graph
over $\Bc_\delta\cap\Sc$. 
By Rademacher's theorem, this justifies the
treatment of  $M$ as a manifold.

\begin{theorem}
	\label{thm6.1}
	The map $\Phi=\Phi$ in \defnref{Phi} satisfies
	\begin{align}
		\label{6.1}
		\norm{\Phi(\eta_0)}_s&\ls \norm{\eta_0}_s^2\quad (\eta_0\in \Bc_\delta\cap\Sc),\\
		\label{6.2}
		\norm{\Phi(\eta_0)-\Phi(\eta_1)}&\ls 
		\delta
		\norm{\eta_0-\eta_1}\quad (\eta_0,\,\eta_1\in \Bc_\delta\cap\Sc).
	\end{align}
\end{theorem}

\begin{proof}
	1. We first prove \eqref{6.1}.
	Take $\eta_0\in \Bc_\delta\cap\Sc$ with $\norm{\eta_0}\sim \delta$.
	Fix some $0<\eps\ll1$ to be determined. As in the last part of the 
	proof of \thmref{thm5.1},
	for sufficiently large $M$ and all $m\ge M$, 
	 one can iterate the map $\Psi=\Psi(\cdot,\eta_0)$ for 
	$m$ times to get some $U^{(m)}\in \Ac_\delta$ (this space does not depend on the choice of $\eta_0$) and $\beta_i^{(m)},\,\g_{ij}^{(m)}$
	s.th. the following holds:
	\begin{enumerate}
		\item $\abs{\beta_i^{(m)}-\beta_i}+\abs{\g_{ij}^{(m)}-\g_{ij}}<\eps;$
		\item $\norm{\xi^{(m+1)}(\tau)-\xi^{(m)}(\tau)}<\eps\br{\tau}^{-2}$ for all $\tau\ge0$; 
		\item $\beta_i^{(m)}=-\int_0^\infty f_i(U^{(m)})\lambda^{-1}$ and $\g_{ij}^{(m)}=-\int_0^\infty h_{ij}(U^{(m)})$, where the functionals
		$f_i,\,h_{ij}:\Ac_\delta\to \Rb$ are as in \eqref{4.6}-\eqref{4.7}.
	\end{enumerate} 
	
	Consider the path $U^{(M+1)}$ with $M\gg1$. This solves the Cauchy problem
	\begin{align}
		\label{6.3}
		\dot \xi^{(M+1)}+L(a^{(M)})\xi^{(M+1)}&= -\W (\si^{(M)})\xi^{(M+1)}- \tilde N(a^{(M)},\xi^{(M)}),\\
		\label{6.4}
		\xi^{(M+1)}(0)&= \eta_0+\beta_i^{(M)} \Si^{(1,0)(i,0,0)}(a_0)+\g_{ij}^{(M)}\Si^{(2,0),(i,j,0) }(a_0).
	\end{align}
	For $\norm{\eta_0}_s\sim \delta$ and $\eps,\,\delta\ll1$, by \eqref{4.1}-\eqref{4.1'},
	we have $\norm{\xi^{(M+1)}(0)}_s\ls \norm{\eta_0}_s+\eps$.
	On the other hand, by the decay condition \eqref{3.2}, 
	we find $\norm{\xi^{(M+1)}(\tau)}\ls \norm{\xi^{(M+1)}(0)}_s\br{\tau}^{-2}$.
	Since $\abs{f_i(U^{(M)})},\,\abs{h_{ij}(U^{(M)})}\ls \norm{\xi^{(M)}}_s^2$ by the nonlinear estimate \eqref{A.3} (recall the definition from \eqref{4.6}-\eqref{4.7}), with Items 1-3 above, we find
	\begin{equation}\label{6.5}
		\begin{aligned}
		\abs{\beta_i}&< \eps+\abs{\beta_i^{(M)}}  \\
		&\ls \eps +\int_0^\infty \norm{\xi^{(M)}}_s^2\\
		&< \eps +\int_0^\infty (\norm{\xi^{(M+1)}(\tau')}_s+\eps\br{\tau'}^{-2})^2\,d\tau'\\
		&\ls \eps +\norm{\eta_0}_s^2\int_0^\infty \Bigl(1+\eps\delta^{-1}\Bigr)^2\br{\tau'}^{-4}\,d\tau',
	\end{aligned}
	\end{equation}
	and a similar estimate for $\g_{ij}$. For fixed $\delta$, taking $\eps\to0$ in estimate \eqref{6.5} and the one for $\g_{ij}$ gives \eqref{6.1}.
	
	2. Next, we prove \eqref{6.2}. As discussed in \remref{remP}, 
	the map $\Psi$ depends on the choice of $\eta_0$. Now, consider the map
	$$\Psi:\Ac_\delta\times (\Bc_\delta\cap \Sc)\to \Ac_\delta.$$
	With the parameter space $\Bc_\delta\cap \Sc$ in place
	of the set $B\subset Y$ in \lemref{lemB.2},
	we see that \eqref{6.2} follows from \eqref{5.21} and \eqref{B.6} if we can show
	the following uniform Lipschitz estimates:
	\begin{align}
		\label{6.6'}
		\abs{\Phi(\eta_0)-\Phi(\eta_1)}&\ls \delta \lim_{M\to\infty}\norm{\Psi^{(M)}(U,\eta_0)-\Psi^{(M)}(U,\eta_1)}\\
		\label{6.6}\norm{\Psi(U,\eta_0)-\Psi(U,\eta_1)}&\ls \norm{\eta_0-\eta_1}_s,
	\end{align}
	for  all 
	$\eta_0,\,\eta_1\in\Bc_\delta\cap \Sc,\,U\in\Ac_\delta.$
	Here the notation $\Psi^{(M)}$ means iterating the map $M$ times.
	Notice that the limit in \eqref{6.6'} exists, since by 	\thmref{thm5.1}, the sequences  $\Psi^{(M)}(U,\eta_m),\,m=0,1$
	both converge as $M\to\infty$ in the space $\Ac_\delta$.

	To get \eqref{6.6'}, we use a similar argument as in Step 1.
		Here we only show $\abs{\beta_i(\eta_1)-\beta_i(\eta_0)}\ls \delta
		\norm{\eta_1-\eta_0}_s,$ as the argument to bound the $\g_{ij}$-difference
		is identical. 
		
	Indeed, define $\beta_i(\eta_m),\,\beta_i^{(M)}(\eta_m)$ 
	by adapting the construction from Step 1 for $m=0,1$.
	Then we get the estimate 
	\begin{equation}\label{6.6.1}
		\begin{aligned}
			\abs{\beta_i(\eta_1)-\beta_i(\eta_0)}&<2\eps+\abs{\beta_i^{(M)}(\eta_1)-\beta_i^{(M)}(\eta_0)}\\
			&\ls \eps+\int_0^\infty \abs{f_i(U^{(M)}(\eta_1))-f_i(U^{(M)}(\eta_0))},
		\end{aligned}
	\end{equation}
	for any fix $\eps>0$ and all sufficiently large $M$.
	Here the paths $U^{(M)}(\eta_m),\,m=0,1$ are obtained by iterating
	the map $\Psi(\cdot,\eta_m)$ on some fixed $U\in\Ac_\delta$.
	
	The claim now is that there holds the difference estimate
	\begin{equation}\label{6.6.2}
		\abs{f_i(U^{(M)}(\eta_1))-f_i(U^{(M)}(\eta_0))}\ls \delta\br{\tau}^{-2}\norm{U^{(M)}(\eta_1)-U^{(M)}(\eta_0)},
	\end{equation}
	for all sufficiently large $M$.
	Indeed, if \eqref{6.6.2} holds, then plugging this into \eqref{6.6.1}
	and taking $\eps\to 0,\,M\to\infty$ gives \eqref{6.6'}.

	In view of the definition of $f_i$ from \eqref{4.6} and the Lipschitz
	estimate \eqref{C.2}, to get \eqref{6.6.2},
	it suffices to show that the paths 
	$U^{(M)}(\eta_1),\,U^{(M)}(\eta_0)$ remains uniformly 
	close in $\Ac_\delta$ for all large $M$.
	 This follows from the definition $U^{(M)}(\eta_m)=\Psi^{(M)}(U,\eta_m)$ and 
	 the uniform bound \eqref{6.6}, provided $\eta_0$ 
	 and $\eta_1$ are close.

	3. It remains to prove \eqref{6.6}. To this end, write $(\si^m,\xi^m):=\Psi(U,\eta_m),\,m=0,1$ for some fixed path $U\in\Ac_\delta$. 
	Notice the difference from \thmref{thm5.1}: Here we fix some $U$ but take
	different $\eta_m$, whereas in \thmref{thm5.1} it is the other way around.
	
	By the definition of $\Psi$
	from \eqref{3.9}, we find that the difference $\xi^1-\xi^0$ satisfies
	the Cauchy problem
	\begin{align}
		\label{6.7}
	&\di_\tau(\xi^1-\xi^0)+L(a)(\xi^1-\xi^0)=-\W(\si)(\xi^1-\xi^0),\\ &\di_\tau(\si^1-\si^0)=\vec F(\si)(\xi^1-\xi^0),\label{6.7'}\\
	\label{6.8}
	&(\si^1-\si^0, \xi^1-\xi^0)\vert_{\tau=0}=(0,\eta_1-\eta_0).
	\end{align}
	Here $\si$ and $a$ are components of the fixed path $U$,
	and the r.h.s. of \eqref{6.7} is bounded by \eqref{A.2}.
	
	Now we study \eqref{6.7}-\eqref{6.8}
	as in the proof of \thmref{thm5.1}. Indeed, if we decompose $\xi^1(\tau)-\xi^0(\tau)$ in terms of the eigenfunctions of $L(a(\tau))$,
	then the projections of this difference along the stable modes decay
	exponentially by the evolution \eqref{6.7}, and those along the zero-unstable
	modes decay at the rate of $\br{\tau}^{-2}$ by the construction
	from \thmref{thm4.1}. The conclusion is the estimate 
	$\norm{\xi^1(\tau)-\xi^0(\tau)}_s\ls \br{\tau}^{-2}\norm{\xi^1(0)-\xi^0(0)}=\br{\tau}^{-2}\norm{\eta_1-\eta_0}$. 
	Plugging this into \eqref{6.7'} and integrating, we get the estimate $\abs{\si^{1}-\si^{0}}\ls \br{\tau}^{-1}\norm{\eta_1-\eta_0}$
	(here we use the zero initial condition from \eqref{6.8}).
	Thus \eqref{6.6} follows.

\end{proof}

\section*{Acknowledgment}
The Author is supported by Danish National Research Foundation grant CPH-GEOTOP-DNRF151. The Author thanks NM M\o ller and JMS Ma for  helpful discussions. 

\section*{Declarations}
\begin{itemize}
	\item Conflict of interest: The Author has no conflicts of interest to declare that are relevant to the content of this article.
\end{itemize}

\appendix
\section{Elementary Lemmas}
\label{sec:B}
In this section we prove some abstract and  elementary lemmas.
Despite their simple appearcance, they play crucial roles in 
Sects. \ref{sec:4}-\ref{sec:5}.

Recall that the weighted Sobolev space $X^s(a)$ is defined in \secref{sec:2}.
As discussed in that section, for two numbers
$b\le 1/2\le a$ we have the trivial embedding
$$X^s(b)\le X^s\equiv X^s(1/2)\subset X^s(a)\quad (s\ge0).$$
In Sect. 4, we are concerned with the $X^s$-estimate
for a family of functions $\phi(t)\in X^s(b)$,
knowing some estimate of the varying norm $\norm{\phi(t)}_{s,a(t)}$ for
a function $a(t)\ge1/2$. 
\lemref{lemB.0} below allows us to deal with this issue.

\begin{definition}[pivot condition]\label{pivot}
	Fix some $0<\delta<1/8$.
	A family of functions $\phi(\cdot, t):\Rb^{n-k}\times \Sb^k\to \Rb,\,t\ge0$
	satisfied \textit{the pivot condition} if there is constant $c>0$ independent of $t$ s.th. 
	\begin{equation}
		\label{B.7}
		\int \sum_{\abs{\al}\le s}\abs{\di^\al \phi(t)}^2 \rho_b\le c, \quad b:=\frac{1}{2}-4\delta>0.
	\end{equation}
	Equivalently, this means $\phi(t)\in X^s(b)$ and is uniformly bounded for all $t$ in the $X^s(b)$-norm. 
\end{definition}

\begin{lemma}\label{lemB.0}
	Let $a(t)$ be a function satisfying $0\le a(t)-1/2\le2 \delta$. 
	Suppose $\phi(t)$ satisfies the pivot condition \eqref{B.7}, and 
	 $\norm{\phi(t)}_{s,a(t)}\le g(t)$ for some positive function $g$.
	Then there holds 
	\begin{equation}
		\label{B.8}
		\norm{\phi(t)}_s\le c^{1/12} g(t)^{2/3}.
	\end{equation}
	
\end{lemma}
\begin{proof}
	Let $u:=\sum_{\abs{\al}\le s}\abs{\di^\al \phi}^2$.
	Then $u\rho_{1/2}=(u^{2/3}\rho_{2a/3})(u^{1/3}\rho_{\frac{1}{2}-\frac{2a}{3}}).$
	By H\"older's inequality with $p=3/2$ and $q=3$, we find
	\begin{equation}\label{B.8'}
		\norm{\phi}_s^2=\int u\rho_{1/2}\le \del{\int u \rho_{a}}^{2/3}\del{\int u\rho_{\frac{3}{2}-2a}}^{1/3}=\norm{\phi}_{s,a}^{4/3}\del{\int u\rho_{\frac{3}{2}-2a}}^{1/3}.
	\end{equation}
	The assumption $0\le a-1/2\le 2\delta<1/4$ implies $\tfrac{3}{2}-2a\ge b>0$, and therefore
	the second factor in the r.h.s. of \eqref{B.8'} is bounded
	by the l.h.s. of \eqref{B.7}. Hence \eqref{B.8} follows from the assumption
	$\norm{\phi}_{s,a}\le g$ and taking square root on both sides of \eqref{B.8'}.
\end{proof}
\begin{remark}
	More generally, for all $1< p<\infty$, one can vary the pivot
	condition to obtain
	similar estimate as \eqref{B.8} with higher power in the r.h.s..
	For our need, any power $1<p<2$ will do.
\end{remark}

By \lemref{lemB.0}, every dissipative estimate in terms of the
$X^s(a)$-norm with possibly changing $a\ge 1/2$ 
gives a (weaker) dissipative estimate in the fixed $X^s$-norm.
This is important for many estimates in Sect.\ref{sec:4}.

The following two lemmas are the mechanism behind \thmref{thm4.1} and \thmref{thm5.1}, respectively. These results are obtained in \cite{MR2480603}.

The first lemma concerns with Cauchy problem for an inhomogeneous ODE.
\begin{lemma}[c.f. \cite{MR2480603}*{Lem. 23}]
	\label{lemB.1}
	Fix two functions $a(t)\ge c>0$ and $f\in L^1([0,\infty),\Rb)$.
	Consider the Cauchy problem for  $x:[0,\infty)\to \Rb$:
	\begin{align}
		&\dot x-a(t)x=f,\label{B.1}\\
		&x(0)=x_0\in \Rb\label{B.2}.
	\end{align}
	There exists a unique bounded solution if and only if 
	\begin{equation}
		\label{B.3}
		x_0=-\int_0^\infty f_i(t')e^{-\int_0^{t'} a}\,dt'\text{ in \eqref{B.2}}.
	\end{equation}
	
	Moreover, if \eqref{B.3} holds, then the solution to \eqref{B.1}-\eqref{B.2} is given by
	\begin{equation}
		\label{B.4}
		x(t)=-\int_t^\infty f(t')e^{-\int_t^{t'} a},\,d\tau'.
	\end{equation}
\end{lemma}
\begin{proof}
	The variation of parameter formula gives the general solution to \eqref{B.1}-\eqref{B.2} as
	\begin{equation}
		\label{B.5}
		x(t)=e^{\int_0^t a} (x_0+\int_0^t f e^{-\int_0^{t'} a}\,dt').
	\end{equation}
	Multiplying both sides of \eqref{B.5} by 
	$e^{-\int_0^t a}$, and then taking $t\to \infty $, 
	we find that if $x(t)\le C$, then  $x(t)e^{-\int_0^t a}\to 0$ as $t\to \infty$ because of the uniform bound $a(t)\ge c>0$,
	and therefore \eqref{B.3} holds.
	
	Conversely, if \eqref{B.3} holds, then the general formula \eqref{B.5} simplifies to \eqref{B.4}.
	For $a(t)>0$ and $f\in L^1$  we can conclude from here that $x(t)\le \norm{f}_{L^1}$.

\end{proof}

For the next lemma, we adopt similar  notations as in Sects. \ref{sec:4}-\ref{sec:6}.
\begin{lemma}[c.f. \cite{MR2480603}*{Lem. 28}]
	\label{lemB.2}
	Let $A\subset X$ be a closed subset of a Banach space $(X,\norm{\cdot})$. Let $B\subset Y$ 
	be a subset of a normed vector space $(Y,\norm{\cdot}_Y)$.
	Let
	 $\Psi:X\times Y\to X$
	be a map s.th. 
	\begin{enumerate}
		\item $\Psi(A\times B)\subset A$;
		\item There exists $0<\rho<1$ s.th.
		$$\sup_{\eta\in B}\norm{\Psi(u^{(1)},\eta)-\Psi(u^{(0)},\eta)}\le \rho\norm{u^{(1)}-u^{(0)}}$$ for two vectors $u^{(0)},u^{(1)}\in A$;
		\item There exists $\al,\delta>0$ s.th.
		$$\sup_{u\in A}\norm{\Psi(u,\eta_1)-\Psi(u,\eta_0)}\le \delta \norm{\eta_1-\eta_0}_Y^\al$$
		for two vectors $\eta_1,\eta_0\in B$.
	\end{enumerate}
	
	Then for every $\eta\in B$, there exists a unique fixed point $u(\eta)$ s.th. $\Psi(u(\eta),\eta)=u(\eta)$.
	Moreover, there holds
	\begin{equation}\label{B.6}
		\norm{u(\eta_1)-u(\eta_0)}\le \frac{\delta}{1-\rho}\norm{\eta_1-\eta_0}_Y^\al.
	\end{equation}
\end{lemma}
\begin{proof}
	Fix $u_*\in A$. For every $\eta\in B$, let $u^{(0)}(\eta):=u_*$, $u^{(m)}(\eta)=\Psi(u^{(m-1)}(\eta),\eta)$ for $m=1,2,\ldots,$ and $u(\eta):=\lim_{m\to\infty} \Psi(u^{(m)}(\eta),\eta)$. This limit exists
	and is the unique fixed point of $\Psi$ in $A$, by Item 1-2 in the assumptions and 
	the standard contraction mapping theorem. 
	
	For the estimate \eqref{B.6},
	compute for $m=1,2,\ldots$
	\begin{align*}
		\norm{u^{(m)}(\eta_1)-u^{(m)}(\eta_0)}\le& \norm{\Psi(u^{(m-1)}(\eta_1),\eta_1)-\Psi(u^{(m-1)}(\eta_0),\eta_1)}\\&+
		\norm{\Psi(u^{(m-1)}(\eta_0),\eta_1)-\Psi(u^{(m-1)}(\eta_0),\eta_0)}\\
		\le& \rho \norm{u^{(m-1)}(\eta_1)-u^{(m-1)}(\eta_0)}+\delta\norm{\eta_1-\eta_0}_Y^\al\\
		\le& \delta\sum_{n=0}^m\rho^n\norm{\eta_1-\eta_0}_Y^\al.
	\end{align*}
	In the last step we use induction on $m$ and Item 3 in the assumptions. Taking $m\to \infty$ gives \eqref{B.6}. 
\end{proof}

\section{Global well-posedness of a linear Cauchy problem}
\label{sec:A}
In this section we study the global well-posedness of the problem \eqref{3.4}-\eqref{3.7}.
In particular, we show  the key map $\Psi$ introduced in \secref{sec:3.1}
is well-defined.

Indeed, for fixed $(\si^{(0)},\xi^{(0)})$, \eqref{3.8}, which is equivalent 
to the system \eqref{3.4}-\eqref{3.5} as we show below,
is a linear parabolic equation w.r.t. the drift Laplacian on the space $X^s(a)$,
together with a source term as a function of $(\si^{(0)},\xi^{(0)})$.
In view of this, the well-posedness results in this section is rather standard.
The small twist in our situation is that the drift Laplacian in \eqref{3.8} is slowly changing, as the drift parameter $a\sim 1/2$ is a function of time.

	For the main existence and regularity result in this section,
we need the following propagator estimate for the linearized operator $L(a)$.

\begin{lemma}
	\label{lemA.0}
	Fix a function $a(t),\,t\ge0$ satisfying $0\le a(t)-1/2\le2 \delta$. 
	For each $t$, define the linear propagator
	\begin{equation}
		\label{U}
		U(t):=e^{-L(a(t)}:X^{s-2}(a(t))\to X^s(a(t)),
	\end{equation}
	where the operator $L(a)$ is defined in \eqref{2.3}.
	Then for every constant $b\in[1/2-4\delta,1/2]$,
	 there holds 
	\begin{align}
		\label{U1}
		U(t)(X^{s-2}(b))\subset X^s(b),\\
		\label{U2}
		\norm{U(t)}_{s-2,b\to s,b}\le C(t).
	\end{align}
\end{lemma}
\begin{proof}

By the spectral theorem, it suffices to prove the restriction of $L(a(t))$ maps
$X^s(b)\subset X^s(a(t))$ continuously into $X^{s-2}(b)$.

Using the formula \eqref{2.3}, we have the relation $$L(a(t))=L(b)+(b-a(t))V,\quad V:\phi\mapsto (-\inn{y}{\grad_y(\cdot)}+\tfrac{1}{k}\Lap_\om+2).$$
The operator $L(b)$ is continuous from $X^s(b)\to X^{s-2}(b)$.
The operator $V$, which is independent of the parameters $a(t)$ and $b$,
satisfies $\norm{V\phi}_{s-2,b}\le c_{n,k}\norm{\phi}_{s,b}$.
(For this, one uses the estimate $\norm{\inn{y}{\grad_y\phi}}_{s-2,b}\le c_n \norm{\phi}_{s,b}$,
c.f. \cite{MR2289828}*{Sec. 4}.)
Since $\abs{b-a(t)}\le 6\delta$, it follows that for sufficiently small $\delta$ there holds
\begin{equation}\label{A.8}
	\norm{L(a(t))\phi}_{s-2,b}\le (1+\delta^{1/2}) \norm{L(b)}_{s,b\to s-2,b}\norm{\phi}_{s,b},\quad t\ge0.
\end{equation}
This shows the mapping property \eqref{U1} and the estimate \eqref{U2} with $C(t)=e^{(1+\delta^{1/2})t\norm{L(b)}_{s,b\to s-2,b} }.$

\end{proof}


\begin{lemma}
	\label{lemA.1}
	Fix $a_0>1/2$ in \eqref{3.6}. Let $b=\tfrac{1}{2}  -4\delta$ for some
	$0<\delta<1/8$.
	For every $\eta_0\in X^s(b)\cap \Sc$ (see \defnref{defn3.2}),
	every fixed path $(\si^{(0)},\xi^{(0)})\in \Ac_\delta$,
	and any constants $\beta_i,\g_{ij}$ in \eqref{3.7},
	the problem \eqref{3.4}-\eqref{3.7} has 
	a unique global solution in the space
		$$Lip([0,\infty),\Si)\times (C([0,\infty),X^s(b))\cap C^1([0,\infty),X^{s-2}(b))),\quad b=\tfrac{1}{2}-4\delta.$$

	Moreover, for all $\tau\ge0$, the solution satisfies the orthogonality condition (c.f. \eqref{2.14})
	\begin{equation}
		\label{A.0}
					\inn{\xi(\tau)}{\Si^{(m,n)}(a^{(0)}(\tau))}_{a^{(0)}(\tau)},\quad \;(m,n)=(0,0),(0,1),(1,1).
	\end{equation}
\end{lemma}
\begin{proof}
	Once the coefficients are frozen, \eqref{3.4}-\eqref{3.5} are decoupled. Thus, we first solve for $\xi$, and then $\si$ can be solved by integrating \eqref{3.5}.

	Plugging \eqref{3.5} to \eqref{3.6}, we find
	\begin{equation}
		\label{A.1}
		\dot \xi+L(a^{(0)})\xi= -\W (\si^{(0)})\xi- \tilde N(a^{(0)},\xi^{(0)}).
	\end{equation}
	Here, for $\vec F,\,\vec M$ as in \eqref{3.5}, we define
	$\W(\si^{(0)}):=\di_\si W(\si^{(0)})\vec F(\si^{(0)})$,  
	and \begin{equation}\label{tildeN}
		\tilde N(a^{(0)},\xi^{(0)}):= N(a^{(0)},\xi^{(0)})+		\di_\si W(\si^{(0)})\vec M(\si^{(0)},\xi^{(0)}),
	\end{equation} which
	is the sum of the (frozen) nonlinearity in both equations. 
	
	Consider the two terms in the the r.h.s. of \eqref{A.1}.
	The second, nonlinear term is studied in \secref{sec:C} below. 
	In the first term, 
	the linear operator $\W$ satisfies the uniform bound 
	\begin{equation}\label{A.2}
		\norm{\W (\si^{(0)}(\tau))}_{s,b\to s,b}\le c \quad (s\ge 2, b>0,\tau\ge0),
	\end{equation} 
	for some absolute constant $c>0$ (say, $500$).
	To get this, one uses the explicit form of $W(\si)$ (see also \remref{remW}), the decay condition \eqref{3.1} of $\si$,
	and the formulae \eqref{2.17}-\eqref{2.15},
	in which the linear terms in $\xi$ defines $\vec F$.

	Fix some  $T>0$ to be determined.
	Define the space
	$$A:=C([0,T],X^s(b)),\quad b:=\tfrac{1}{2}-4\delta,
	\norm{\phi}_A:=\sup_{0\le \tau \le T} \norm{\phi(\tau )}_{s,b}.$$
	Clearly, this $\norm{\cdot}_A$ turns $A$ into a Banach space.
	
	For any initial configuration $\xi(0)$ of the form \eqref{3.6},  together with
	a path $\phi \in A$,
	consider the map
	\begin{equation}
		\label{A.4}
		F:\phi \mapsto \del{\tau\mapsto U(\tau)\xi(0)-\int _0^\tau U(\tau-\tau') \del{\W(\si^{(0)})\phi(\tau')+ \tilde N(a^{(0)},\xi^{(0)})} \,d\tau'}.
	\end{equation}
	Here $U(\tau):=e^{-L(a^{(0)})\tau }$ is the propagator 
	associated to $L(a^{(0)}(\tau))$.
	By Duhamel's principle, any fixed point of the map $F$ is a (mild) local solution
	to \eqref{A.1}. 
	
	With the choice $a=a^{(0)}(\tau)$ satisfying \eqref{3.1} and $b=1/2-4\delta$,
	the assumption of \lemref{lemA.0} is satisfied. By \eqref{U1}, 
	we see that $F(A)\subset A$. 
	
	Take $\phi_1,\,\phi_2\in A$ and consider the difference 
	$$(F(\phi_1)-F(\phi_2))\vert_\tau=-\int_0^\tau U(\tau-\tau')\W (\si^{(0)}(\tau'))(\phi_1(\tau')-\phi_2(\tau')),\,d\tau'.$$
	By the propagator estimate $\norm{U(\tau)}_{s-2,s}\le C(\tau)$
	in \eqref{U2}, together with \eqref{A.2}, from here we can choose 
	the length $T$ of the existence interval, s.th.
	the map $F$ in \eqref{A.4} is a contraction in $A$.
	This gives a local solution $\xi(\tau),\tau\le T$ to \eqref{A.1}.
	This choice $T$ depends only on the constant $c>0$ in \eqref{A.2}
	and the constant $C(t)$ in the propagator estimate \eqref{U2}.
	(The latter is large as $t$ increases due to the unstable modes,
	and therefore is a constraint for choosing $T$.)
	Since this choice of  $T$ does not depend on the initial condition $\xi(0)$,
	we can iterate this process starting from $\xi(T)$ to get a local solution
	on $T\le \tau\le 2T$, etc. 
	It follows that 
	we have now a global solution $\xi \in C([0,\infty),X^s(b))$ so long as the initial condition $\xi(0)\in X^s(b)$.
	
	Moreover, by \lemref{lemB.2}, in particular the estimate \eqref{B.6}, 
	we find the fixed point $\xi=\xi(\tau)$ to be locally Lipschitz in $\tau$.
	By Rademacher's theorem, this implies the path $\di_\tau \xi$ exists
	everywhere. Thus we can differentiate the fixed point equation $\xi(\tau)=F(\xi(\tau))$
	w.r.t. $\tau$ to get \eqref{A.1} for a.e. $\tau$. 
	On the other hand, by \eqref{A.8} and the expression \eqref{N} below, 
	for $(\si^{(0)},\xi^{(0)})\in\Ac_\delta$, so long as $\xi(\tau)$ lies in $X^s(b)$,
	the path $$\tau\mapsto L(a^{(0)}(\tau))\xi(\tau)+\W (\si^{(0)}(\tau))\xi(\tau)+\tilde N(a^{(0)}(\tau),\xi^{(0)}(\tau)) \in X^s(b)$$
	 lies  in $X^{s-2}(b)$ and is continuous.
	Thus we see in fact $\xi\in C^1([0,\infty),X^{s-2}(b)).$

	As in \secref{sec:2.3}, once we have solved $\si$ from \eqref{3.5},
	we can conclude the orthogonality condition \eqref{A.0}
	as long as it holds at $\tau=0$. 
	The latter is indeed the case, because of the choice for $\eta_0$ as $X^0(a_0)$-perpendicular
	to all the zero-unstable modes of $L(a_0)$,
	and the fact that the remaining terms in \eqref{3.6} already
	satisfy \eqref{A.0} at $\tau=0$, owning to the orthogonality 
	among various eigenfunctions of $L(a_0)$.

\end{proof}

%
Next, we consider the regularity of the solution constructed above.
\begin{lemma}
\label{lemA.2}
	Let $(\si,\xi)$ be the global solution constructed in \lemref{lemA.1}.
	Let $P^{(m,n)}$ be the projection onto the eigenspaces spanned
	by the vectors $\Si^{(m,n)(i,j,l)}(a)$ defined in \eqref{2.4}-\eqref{2.8} above.
	Then, for $s\ge2$, there holds the a priori estimate
\begin{equation}
	\label{A.5}
	\norm{P^{(m,n)}\xi}_{s}\ls \norm{\xi}_{0}.
\end{equation}

Moreover, this solution satisfies the pivot condition
\begin{equation}
	\label{A.7}
	\norm{\xi(\tau )}_{s,b}^2\le c\quad b:=\tfrac{1}{2}-4\delta.
\end{equation}
c.f. \eqref{B.7} and \eqref{3.3'}.

Moreover, the path $\si(\tau)$ satisfies
\begin{equation}
	\label{A.9}
	\abs{\si(\tau)-\si(0)}\ls \tau\sup_{0\le\tau'\le\tau}\norm{\xi(\tau')}_s.
\end{equation}
where $a_0>1/2$ is as in \eqref{3.6}.
\end{lemma}
\begin{proof}

	
	For the unstable parts $\xi^{(m,n)}:=P^{(m,n)}\xi$,
	consider the eigenvalue equation $L(a)\xi^{(m,n)}=\lambda_{m,n} \xi^{(m,n)}$
	in the space $X^s\subset X^s(a)$.
	(This problem is well-posed by \eqref{A.8}.)
	Following \cite{MR3374960}*{Lem. 3.11}, we use Bochner's formula for the drift
	Laplacian
	to get elliptic estimates of the form $\norm{\xi^{(m,n)}}_s\le C_{m,n}\norm{\xi^{(m,n)}}_{s-2}$.
	These imply $ \norm{\sum \xi^{(m,n)}}_s\le (\sum C_{m,n} )\norm{\sum\xi^{(m,n)}}_{s-2}$.
	Thus \eqref{A.5} follows.

	For \eqref{A.7}, consider the equation \eqref{A.1}.
	In \lemref{lemA.1} we have shown the existence of $X^s(b)$-solution $\xi$ 
	to this linear parabolic equation. 
	Standard semi-group theory
	for such flows provides  continuous dependence of the $X^s(b)$-norm of the 
	solution $\xi$ 
	in terms of the $X^{s-2}(b)$-norm of the data in \eqref{A.1}. 
	In view of \eqref{A.8} and \eqref{A.2}, it
	suffices to show the source term  
	$\tilde N(a^{(0)},\xi^{(0)})$ lies in $X^{s-2}(b)$.
	(Here we are not claiming any decay estimates in the $X^s(b)$-norm.)
	
	Consider the two terms in the r.h.s. of 
	$$\tilde N(a^{(0)},\xi^{(0)})= N(a^{(0)},\xi^{(0)})+		\di_\si W(\si^{(0)})\vec M(\si^{(0)},\xi^{(0)}).$$
	By the expression \eqref{N} below, 
	the nonlinear map $N$ is continuous from $X^s(b)\to X^{s-2}(b)$.
	The second term in $\tilde N$ 
	is of the order
	$\norm{N}_{s-2,b}\norm{\xi^{(0)}}_{0,b}$. 
	The conclusion is the continuous dependence
	\begin{equation}
		\label{A.3'}
		\norm{\tilde N(a^{(0)},\xi^{(0)})}_{s-2,b}\le c  \norm{\xi^{(0)}}_{s,b} \quad (\tau\ge0),
	\end{equation}
	where the  constant $c$ in the r.h.s. depends on that in \eqref{A.3} only. Estimate
	\eqref{A.3'}, the assumption that $\xi^{(0)}$ satisfies the pivot condition \eqref{B.7},
	 together with  the continuous dependence mentioned earlier, 
	gives \eqref{A.7}.
	
	Lastly, \eqref{A.9} follows from integrating \eqref{3.5}.
	
\end{proof}

\section{Properties of the Nonlinearity}
\label{sec:C}
In this section we  prove two key estimates of the nonlinear map $N$ in \eqref{2.1'}.

For $a>0$, the nonlinear map $N(a,\cdot):X^s(a)\to X^{s-2}(a)$ is defined in \eqref{2.1'},
as the remainder in the expansion $F'_a(\sqrt{k/a}+\xi)=L(a)\xi+N(a,\xi)$,
which is valid in a neighbourhood around zero in $X^s(a)$ due to the 
$C^2$ regularity of the $F$-functional.
By this definition, 
\eqref{F'}, and \eqref{2.3}, we find that $N(a,\xi)$ is given by
\begin{align}
	\label{N}
		\begin{split}
			N(a,\xi)=&\left(\frac{a}{k}-v^{-2}\right)\Lap_\om \xi
			-av+kv^{-1}+2a\xi\\
			&+N_1(v),
		\end{split}\\
		\label{N1}
	\begin{split}
		N_1(v)=&\frac{v^{-4} \inn{\grad^2_\om v\grad_\om v}{\grad_\om v}-v^{-3}\abs{\grad_\om v}^2}{1+\abs{\grad_yv}^2+v^{-2}\abs{\grad_\om v}^2}\\
		&-\frac{2v^{-2} \inn{\grad_yv}{\inn{\grad_\om v}{\grad_\om}\grad_yv}
			-\inn{\grad^2_yv\grad_yv}{\grad_yv} }{1+\abs{\grad_yv}^2+v^{-2}\abs{\grad_\om v}^2},
	\end{split}
\end{align}
where $v=\sqrt{k/a}+\xi$. This expansion
is obtained in e.g. \cite{MR3397388}*{Chap.1}, \cite{MR4303943}*{Sect.3}. 

As pointed out in \cite{MR3374960}, though 
$N(a,\xi)$ arises as the quadratic remainder 
from the expansion of the $X^0(a)$-gradient $F_a'$,
due to the decay property of the Gaussian measure,
one cannot establish a quadratic estimate
of the form $\norm{N(a,\xi)}_{s-2}\ls \norm{\xi}_s^2$.

However, it is possible to get some quadratic estimate for
 the projection of $N(a,\xi)$ into  specific eigenspaces of $L(a)$, owning to 
 the explicit forms of the latter.
As far as the decay estimates in Sect. \ref{sec:4}-\ref{sec:6}
are concerned, 
it suffices  to bound the projections of $N$ 
onto the span of \eqref{2.6} and \eqref{2.8}.

\begin{lemma}\label{lemC1}
	Fix $0<\delta\ll1$ and suppose $1/2+\delta\le a \le 1/2+2\delta$.
	Let $P^{(m,n)}$ be the projection onto the eigenspaces spanned
	by the vectors $\Si^{(m,n)(i,j,l)}(a)$ defined in \eqref{2.4}-\eqref{2.8} above.
	Then there hold
	\begin{equation}
		\label{A.3}
		\norm{P^{(m,n)}N(a,\xi)}_{s-2,a}\le c\norm{\xi}_s^2,
	\end{equation}
	where the constant $c>0$
	on the r.h.s. depends on dimension only.
\end{lemma}

\begin{proof}
	In  \cite{MR3374960}*{(4.19)}, the following pointwise estimate for $N$
	is obtained:
	\begin{equation}\label{C.0}
	N(a,\xi(y,\om))\ls (1+\abs{y})(\abs{\xi(y,\om)}^2+\abs{\grad\xi(y,\om)}^2+\abs{\grad^2\xi(y,\om)}^2).
	\end{equation}
	By a rescaling argument,
	in principle the implicit constant can depend on $a$.
	But for all $\abs{a-1/2}\ll1$,
	one can make the constant uniform in $a$. 
	
	The claim now is that \eqref{A.3} holds with $(m,n)=(2,0)$.
	Indeed, this gives the worst bound of all such projections,
	and only this projection and the one with $(m,n)=(1,0)$ 
	are needed in the previous sections, and all 
	other estimates follow the same argument.
	
	Let $u=\abs{\xi}^2+\abs{\grad\xi}^2+\abs{\grad^2\xi}^2$. Recall the definition of the function $\Si^{(2,0)(i,j,0)}$
	from \eqref{2.8}.
	For $s=2$, integrating \eqref{C.0} over the cylinder against 
	the Gaussian measure, we find
	\begin{equation}\label{C.0.1}
\begin{aligned}
			\norm{P^{(2,0)}N}_{0,a}&=\sum_{i,j=1}^{n-k}\abs{\int N\Si^{(2,0)(i,j,0)}(a)\rho_a}\\
			&\ls \int (\abs{y}^2+\abs{y}^3)u\rho_a\\
			&\le \sup_{\Rb^{n-k}\times \Sb^k} (\abs{y}^2+\abs{y}^3)\rho_{a-1/2}(y,\om)\int u\rho_{1/2}. 
\end{aligned}
	\end{equation}
	For $a-1/2\ge\delta,$ the prefactor in the last line can be made uniform,
	and one can conclude the claim from here.

For $s\ge3$, one can differentiate the point-wise estimate \eqref{C.0},
replace the second line of \eqref{C.0.1} with the result, and then proceed as above.
\end{proof}

	From the proof above one can see what prevents an
estimate of the form $\norm{N}_{s-2,a}\ls \norm\xi_s^2$,
which is already weaker than an estimate of the form $\norm{N}_{s-2}\ls \norm\xi_s^2$.
Indeed, take an orthonormal basis of the Hilber space 
$X^s(a)$. Then this consists of
a sequence of functions that has
increasing asymptotic growth.
For instance, in the proof of \lemref{lemC1} above,
we consider the basis of $X^s(a)$ given by the 
eigenfunctions of the self-adjoint operator $L(a)$,
whose growth are dominated by the Hermite polynomials.
The latter have growing orders of polynomial growth.
Thus, one cannot proceed from the second to the third line in \eqref{C.0.1}
if one sums over all projections.

The following lemma is needed to derive the Lipschitz 
estimate \eqref{5.9} in \secref{sec:5}.
\begin{lemma}
	For $s\ge2$, define 
	 $$V(\delta_1):=\Set{(a,\xi)\in \Rb\times X^s:\norm{\xi}_s\le \delta_1,0<a<1}$$ for every sufficiently small $0<\delta_1\ll a^{-1/2}$.
	Then for any two paths $(a^{(m)},v^{(m)})\in U(\delta),\,m=1,2$, there holds the Lipschitz estimate 
	\begin{equation}
		\label{C.2}
		\norm{N(a^{(1)},\xi^{(1)})-N(a^{(0)},\xi^{(0)})}_{s-2}\ls\delta_1 \del{\abs{a^{(1)}-a^{(0)}}+\norm{\xi^{(1)}-\xi^{(0)}}_s}.
	\end{equation}
\end{lemma}
\begin{proof}
	Consider the trivial inequality
	\begin{multline}\label{C.3}
		\norm{N(a^{(1)},v^{(1)})-N(a^{(0)},v^{(0)})}_{s-2}\\\ls 
		\norm{N(a^{(1)},v^{(1)})-N(a^{(1)},v^{(0)})}_{s-2}+\norm{N(a^{(1)},v^{(0)})-N(a^{(0)},v^{(0)})}_{s-2}.
	\end{multline}
	We bound the two terms in the r.h.s. respectively.

	To bound the first term, consider the explicit expression \eqref{N} above.
	From here we see that for fixed $a$, the map $N(a,\cdot)$ is a rational function of $v^{-k},\,k=0,\ldots, 4$ and the derivatives $\di^\al_y\di^\beta_\om v,\,
	\abs{\al}+\abs{\beta}\le 2$.
	Hence, for $s\ge 2$, 
	the map $N(a,\cdot):V(\delta_1)\to X^{s-2}$ is a smooth
	 by the uniform bound \eqref{A.3}
	proven earlier. 
	The claim now is that  
	the partial Fr\'echet derivative $\di_\xi N(a,0): X^s\to X^{s-2}$ vanishes.
	If this holds, then 
	we have
	\begin{equation}\label{C.4}
		\norm{N(a^{(1)},v^{(1)})-N(a^{(1)},v^{(0)})}_{s-2}\ls \delta_1 \norm{\xi^{(1)}-\xi^{(0)}}_s,
	\end{equation}
	which follows from the mean value theorem for Fr\'echet
	differentiable maps, c.f. \cite{MR1225101}*{Thm. 1.8}.
	
	To see $\di_\xi N(a,0)=0$, we differentiate the expansion $N=F'_a-L(a)$ w.r.t. $\xi$
	to obtain
	$\di_\xi N(a,\cdot)=dF_a'(\cdot)-L(a)$, where $dF_a'$ 
	is the Hessian
	of the $F$-functional defined in \eqref{F}. By the $C^2$ regularity of $F_a$, 
	(say, from the analysis of  the Hessian in \cite{MR2993752}*{Sect. 4}),
	one can evaluate this expression at a vanishing sequence of $\xi_n$
	in $X^s$ to get the claim. Thus \eqref{C.4} follows.

	To bound the second term in the r.h.s. of  \eqref{C.3}, 
	we need to use the expression \eqref{N}.
	The point here is that for fixed $\xi$ with $\norm{\xi}_2\ll1$,
	the leading order term of $N(\cdot, \xi)$ lies in the first five terms from \eqref{N},
	namely 
	\begin{equation}
		\label{C.5}
		\left(\frac{a}{k}-v^{-2}\right)\Lap_\om \xi
		-av+kv^{-1}+2a\xi.
	\end{equation}
	The claim is that the $X^{s-2}$-norm
	of  this expression can be bounded by some product of a Lipschitz
	function in $a$ and a function of the order $O(\norm{\xi}_s)$.
	If this the case, then it follows from the assumption $\norm{\xi}\le \delta_1$
	that 
		\begin{equation}\label{C.6}
		\norm{N(a^{(1)},v^{(0)})-N(a^{(0)},v^{(0)})}_{s-2}\ls \delta_1\abs{a^{(1)}-a^{(0)}}.
	\end{equation}

	To see \eqref{C.5} has the claimed property, we find
	\begin{align*}
		\norm{\left(\frac{a}{k}-v^{-2}\right)\Lap_\om \xi}_{s-2}&\ls \norm{\xi}_s a,\\
		\norm{-av+kv^{-1}+2a\xi}_s&\ls (\norm{\xi}_s +\norm{\xi}_s ^2)(a+a^{3/2}).
	\end{align*}
	This follows by plugging $v=\sqrt{k/a}+\xi$ into \eqref{C.5} and rearranging. 
	For the second line above to be Lipschitz, we need the assumption $0<a<1$.
	
	Combining \eqref{C.3}-\eqref{C.5} gives \eqref{C.2}.

\end{proof}

\begin{proof}[Proof of \eqref{5.9}]
			Now, for paths $U^m\in \Ac_\delta,\,m=1,2$, 
	define two paths $(a^{(m)},\xi^{(m)})$ in the obvious way. 
	By \eqref{3.2},
	at each $\tau$ we find $$(a^{(m)},\xi^{(m)})\in V(\delta\br{\tau}^{-2}).$$
	This gives \eqref{5.9}.
\end{proof}

		By the same argument, using \eqref{A.3'}, one can get a similar estimate as \eqref{5.9} for the nonlinear map $\tilde N$ in \eqref{tildeN}.

\bibliography{bibfile}
\end{document}